\documentclass[10pt,reqno]{amsart}
\usepackage{amsmath,amssymb,amsfonts,latexsym,amsthm}

\usepackage{cite} 

\usepackage{mathrsfs}

\usepackage{stmaryrd}

\numberwithin{equation}{section}

\usepackage{graphicx,subfigure}
\usepackage{xcolor}
\usepackage{bm}

\usepackage[colorlinks=true,urlcolor=blue,
citecolor=red,linkcolor=blue,linktocpage,pdfpagelabels,
bookmarksnumbered,bookmarksopen]{hyperref}

\usepackage{ifthen} 

\theoremstyle{plain}

\newtheorem{lemma}{Lemma}[section]
\newtheorem{theorem}[lemma]{Theorem}

\newtheorem{proposition}[lemma]{Proposition}

\theoremstyle{definition}

\newtheorem{remark}[lemma]{Remark}
\newtheorem{definition}[lemma]{Definition}
\newtheorem{assumption}[lemma]{Assumption}

\theoremstyle{remark}

\newcommand{\Id}{\mathbb{I}}
\newcommand{\Idd}{\mathrm{Id}}

\newcommand{\bv}{\mathbf{v}}

\newcommand{\A}{\mathbf{A}}

\newcommand{\bw}{\mathbf{w}}

\newcommand{\R}{\mathbb{R}}
\newcommand{\C}{\mathbb{C}}
\newcommand{\Z}{\mathbb{Z}}
\newcommand{\N}{\mathbb{N}}

\newcommand{\cT}{{\mathcal{T}}}

\newcommand{\cL}{{\mathcal{L}}}
\newcommand{\cD}{{\mathcal{D}}}
\newcommand{\cR}{{\mathcal{R}}}
\newcommand{\cS}{{\mathcal{S}}}

\newcommand{\cA}{{\mathcal{A}}}

\newcommand{\cJ}{{\mathcal{J}}}

\newcommand{\ccC}{\mathscr{C}}
\newcommand{\ccB}{\mathscr{B}}
\newcommand{\ccM}{\mathscr{M}}

\renewcommand{\Re}{\mathrm{Re}\,} 
\renewcommand{\Im}{\mathrm{Im}\,}

\newcommand{\ind}{\mathrm{ind}\,}

\newcommand{\nul}{\mathrm{nul}\,}
\newcommand{\codim}{\mathrm{codim}\,}

\newcommand{\bbc}{\overline{c}}

\newcommand{\ess}{\sigma_\mathrm{\tiny{ess}}}
\newcommand{\ptsp}{\sigma_\mathrm{\tiny{pt}}}

\newcommand{\iptsp}{\widetilde{\sigma}_\mathrm{\tiny{pt}}}

\newcommand{\sppi}{\sigma_\pi}
\newcommand{\spd}{\sigma_\delta}

\newcommand{\<}{\langle}
\renewcommand{\>}{\rangle}

\begin{document}

\title[Spectral stability of diffusion-degenerate Nagumo fronts]{Spectral stability of monotone traveling 
fronts for reaction diffusion-degenerate Nagumo equations}

\author[J. F. Leyva]{J. Francisco Leyva}
 
\address{{\rm (J. F. Leyva)} Facultad de Ciencias de la Computaci\'{o}n\\Benem\'{e}rita Universidad Aut\'{o}\-no\-ma de Puebla\\Av. San Claudio y 14 Sur, Ciudad Universitaria, C.P. 72570, Puebla, Puebla (Mexico)}

\email{jfleyva.84@gmail.com}

\author[L. F. L\'{o}pez R\'{\i}os]{Luis F. L\'{o}pez R\'{\i}os}
 
\address{{\rm (L. F. L\'{o}pez R\'{\i}os)} Instituto de 
Investigaciones en Matem\'aticas Aplicadas y en Sistemas\\Universidad Nacional Aut\'onoma de 
M\'exico\\Circuito Escolar s/n, Ciudad Universitaria, C.P. 04510, Cd. de M\'{e}xico (Mexico)}

\email{luis.lopez@iimas.unam.mx}

\author[R. G. Plaza]{Ram\'on G. Plaza}

\address{{\rm (R. G. Plaza)} Instituto de 
Investigaciones en Matem\'{a}ticas Aplicadas y en Sistemas\\Universidad Nacional Aut\'onoma de 
M\'exico\\Circuito Escolar s/n, Ciudad Universitaria, C.P. 04510, Cd. de M\'{e}xico (Mexico)}

\email{plaza@mym.iimas.unam.mx}

\begin{abstract}
This paper establishes the spectral stability of monotone traveling front solutions for reaction-diffusion equations where the reaction function is of Nagumo (or bistable) type and with diffusivities which are density dependent and degenerate at zero (one of the equilibrium points of the reaction). Spectral stability is understood as the property that the spectrum of the linearized operator around the wave, acting on an exponentially weighted space, is contained in the complex half plane with non-positive real part. The degenerate fronts studied in this paper travel with positive speed above a threshold value and connect the (diffusion-degenerate) zero state with the unstable equilibrium point of the reaction function. In this case, the degeneracy of the diffusion coefficient is responsible of the loss of hyperbolicity of the asymptotic coefficient matrices of the spectral problem at one of the end points, precluding the application of standard techniques to locate the essential spectrum. This difficulty is overcome with a suitable partition of the spectrum, a generalized convergence of operators technique, the analysis of singular (or Weyl) sequences and the use of energy estimates. The monotonicity of the fronts, as well as detailed descriptions of the decay structure of eigenfunctions on a case by case basis, are key ingredients to show that all traveling fronts under consideration are spectrally stable in a suitably chosen exponentially weighted $L^2$ energy space.
\end{abstract}



\maketitle

\setcounter{tocdepth}{1}

\tableofcontents

\section{Introduction}

This paper studies the stability of monotone traveling fronts for diffusion-dege\-ne\-rate Nagumo equations of the form
\begin{equation}
\label{degRD}
 u_t = (D(u)u_x)_x + f(u),
\end{equation}
where $u = u(x,t) \in \R$, $x \in \R$, $t > 0$, the reaction function is of \emph{Nagumo (or bistable) type} \cite{NAY62,McKe70} and the diffusion coefficient is \emph{degenerate} at $u = 0$, that is, $D(0) = 0$. More precisely, we assume that 
\begin{equation}
\label{hypD}
\begin{aligned}
& D(0) = 0, \;\; \, D(u) > 0 \; \; \text{for all} \, u \in (0,1],\\
& D \in C^4([0,1];\R) \;\; \text{with} \; D'(u) > 0 \; \text{for all} \; u \in [0,1].
\end{aligned}
\end{equation}
As an example we have the quadratic function 
\begin{equation}
\label{Dbeta}
 D(u) = u^2 + b u,
\end{equation}
for some constant $b > 0$, proposed by Shigesada \cite{Shig80,SKT79} to model dispersive forces due to mutual interferences between individuals of an animal population.

The reaction function $f : \R \to \R$ is supposed to be smooth enough and to have two stable 
equilibria at $u=0, u=1$, and one unstable equilibrium point at $u = \alpha \in (0,1)$, that is,
\begin{equation}
\label{bistablef}
	\begin{aligned}
	&f \in C^3([0,1];\R), &\qquad &f(0)=f(\alpha)=f(1)=0,\\
	&f'(0), f'(1)<0,
		&\qquad  &f'(\alpha)>0,\\
	&f(u)>0 \, \textrm{ for all } \, u \in(\alpha,1),
		&\qquad &f(u)<0 \, \textrm{ for all } \, u \in (0,\alpha),\\
	\end{aligned}
\end{equation}
for a certain $\alpha \in (0,1)$. A well-known example is the widely used cubic polynomial
\begin{equation}\label{cubicf}
	f(u)= u(1-u)(u-\alpha),
\end{equation}
with $\alpha \in (0,1)$.

Scalar reaction-diffusion equations with a bistable reaction function appear in different contexts (see, e.g., \cite{FiM77,MeSc04a,AlCa79,McKe70,NAY62} and the many references therein). Following McKean \cite{McKe70} and for simplicity in the nomeclature, in this work we call it the \emph{Nagumo reaction diffusion equation}. In terms of continuous models of the spread of biological populations, for example, reaction functions of Nagumo type
often describe kinetics exhibiting positive growth rate for population densities over a threshold 
value ($u > \alpha$), and decay for densities below such value ($u < \alpha$). The former is often 
described as the \textit{Allee effect} \cite{MurI3ed}. The function $f$, which underlies two competing stable states, $u=0$ and $u=1$, can also be interpreted as the derivative of a double-well
potential, $F(u) = - \int^u f(s) \, ds$, $F(0) = 0$, with wells centered at those states.

In all these contexts, traveling front solutions play a prominent role. There is a vast literature on the analysis of reaction-diffusion fronts which we will not review. In early works the diffusion coefficient (or motility) was considered as a positive constant. It is now clear that, in many situations in physics and biology, the diffusion coefficient must be a function of the density. In particular, a density-dependent nonlinear diffusion coefficient is called \emph{degenerate} if it approaches zero when the density tends to certain equilibrium points of the reaction (typically zero). In terms of population dynamics, zero diffusion for zero densities is tantamount to requiring no motility in regions of space where the population is very scarce or near absent. In this paper, under hypotheses \eqref{hypD}, we assume that the nonlinear diffusion function $D = D(u)$ vanishes only at $u = 0$; in particular, the condition that $D > 0$ for $u > 0$ means that populations tend to ``avoid crowds" (cf. \cite{Aron85}). Models with degenerate diffusivities are also endowed with interesting mathematical properties. Among the new features one finds the emergence of traveling waves of ``sharp" type (see, for example, \cite{SaMa94a,SaMa97,Sh10}) and, notably, that solutions may exhibit finite speed of propagation of initial disturbances, in contrast with the strictly parabolic case (see, e.g., \cite{GiKe96}). In chemical engineering, the very well-known choice $D(u) = m u^{m-1}$, $m \geq 1$, is often used to model diffusion in porous media (see \cite{Muskat37,Vaz07} and the references therein).

The existence of fronts for reaction-diffu\-sion equations with degenerate diffusion was first studied for very specific forms of the nonlinear diffusion function (for an abridged list of references see \cite{Aron85,ARR81,New80,New83}), particularly in the case of a ``porous medium'' type of diffusion function. In the case of generic degenerate diffusion functions satisfying hypotheses \eqref{hypD}, Sanchez-Gardu\~no and Maini proved the existence of traveling fronts for kinetics of Fisher-KPP (also known as monostable) type \cite{SaMa94a,SaMa95} and of Nagumo (bistable) type \cite{SaMa97}. The authors apply a dynamical systems approach to establish the existence of heteroclinic connections. In the Fisher-KPP model with degenerate diffusion, for example, the degeneracy is responsible for the emergence of sharp traveling solutions for a particular value of the speed and a whole family of smooth traveling fronts for any larger speed (for details, see \cite{SaMa95}). These orbits connect the only two equilibrium points of the Fisher-KPP reaction (namely, $u = 0$ and $u = 1$). In the Nagumo case, the dynamics is much richer, due to both the degeneracy of the diffusion and the presence of a third (unstable) equilibrium point $u = \alpha$. This leads to a wider range of possible homoclinic and heteroclinic connections compared to those of the Fisher-KPP case. The authors in \cite{SaMa97} prove the existence of a unique positive wave speed, $c_* \in (0, \bbc(\alpha))$, where $\bbc(\alpha) := 2 \sqrt{D(\alpha)f'(\alpha)}$, associated to fronts of sharp type. Depending on conditions relating $D$ with $f$, there exists a continuum of monotone fronts, pulses and oscillatory waves. We are interested in  \emph{monotone and smooth} (non-sharp) traveling fronts which can be classified into three types (see Theorem 1 in \cite{SaMa97} or Proposition \ref{propstructure} below):
\begin{itemize}
\item[(I)] \emph{diffusion-degenerate traveling fronts}: for each speed value above a threshold speed, $c > \bbc(\alpha) > 0$, there exists a diffusion-degenerate, monotone increasing front connecting $u = 0$ with $u = \alpha$, traveling with that speed $c$;
\item[(II)] \emph{non-degenerate traveling fronts}: for each value $c > \bbc(\alpha)$ there exists a non-degenerate, monotone decreasing front connecting $u = 1$ with $u = \alpha$; and,
\item[(III)] \emph{stationary diffusion-degenerate fronts}: if $\int_0^1 D(u) f(u) \, du = 0$ then there exists a unique stationary monotone increasing front (with speed $c = 0$) connecting $u = 0$ with $u = 1$ (the two stable equilibria of the reaction), and a unique stationary, monotone decreasing front connecting $u = 1$ with $u = 0$.
\end{itemize}

This paper addresses the stability properties of the first type of waves described above. Our analysis establishes the first step of a general stability program: the property of \emph{spectral stability}, which is based on the analysis of the spectrum of the linearized differential operator around the wave. Spectral stability can be formally defined as the ``well-behavior" of the linearized operator around the traveling front, in the sense that there are no eigenvalues with positive real part which could render exponentially growing-in-time solutions to the linear problem (for the precise statement, see Definition \ref{defspecstab} below). The degeneracy of the diffusion coefficient, however, is responsible of some technical difficulties even at the spectral level. It is to be observed that the spectral stability of diffusion-degenerate Fisher-KPP fronts was already studied in a companion paper \cite{LeP20}. There are new aspects, however, which are particular to the Nagumo case. Here we summarize the contributions of this work:
\begin{itemize}
\item Due to the degeneracy of the diffusion coefficient at $u=0$, one of the end points of the front, the corresponding asymptotic coefficient matrix (when the spectral problem is written as a first order system) ceases to be hyperbolic. This prevents us to apply standard results (basically, the relation between hyperbolicity, exponential dichotomies and the Fredholm borders location \cite{AGJ90,KaPro13,San02}) to control the essential spectrum. For that purpose, we use the particular partition of spectrum introduced in \cite{LeP20} to deal with degenerate problems: instead of the standard  Weyl partition into essential and point spectra, we write the spectrum as the union of three disjoint components, $\sigma = \ptsp \cup \sppi \cup \spd$ (see Defintion \ref{defspecd} below). 
\item We generalize the method in \cite{LeP20} in order to show that there is convergence in the generalized sense of a family of operators (parabolic regularizations of the linearized operator around the front) when the regularization parameter tends to zero. The new partition of spectrum sorts out the complex values $\lambda$ for which the operator $\cL - \lambda$ has closed range or not. Thus, this technique only allows to control a subset of the essential spectrum characterized by this closed range property, namely $\spd$.
\item To locate points in a subset, $\sppi$, of the compressed spectrum (the range is no longer closed) we analyze the behavior of singular (or Weyl) sequences. Appropriately chosen exponential weights allow to control both the ``closed range" points of the essential spectrum and the singular sequences associated to complex values in $\sppi$.
\item These exponential weights allow us, in turn, to perform energy estimates on spectral equations that help to locate the \emph{isolated} point spectrum, $\ptsp$. The estimates are possible thanks to the monotonicity of the fronts and to the fact that the exponential weights imply that a certain transformation remains in the space domain of the operator.
\end{itemize}

It is important to remark that we restrict ourselves to the analysis of degenerate fronts of type I only because, (a), the spectral stability study of non-degenerate fronts is quite standard from the technical viewpoint, inasmuch as it reduces to a well-known spectral problem of Sturm-Liouville type (see, e.g., \cite{BCS20}); and, (b), the spectral stability of monotone stationary fronts can be easily addressed with the techniques introduced here: since $c=0$ there is no need of exponentially weighted spaces and the energy estimates considerably reduce to a particular case of the analysis presented here. The \emph{nonlinear} stability of such stationary fronts, however, will be addressed in a companion paper \cite{FLP}.

\subsection{Main result}

Let us now state the main result of this paper, which guarantees the spectral stability in a suitably chosen exponentially weighted energy $L^2$-space of all fronts in the family of degenerate fronts connecting $u = 0$ with the unstable equilibrium $u = \alpha$.

\begin{theorem}[spectral stability of degenerate non-stationary monotone Nagumo fronts]
\label{mainthmNd}
Un\-der hypotheses \eqref{hypD} and \eqref{bistablef}, the family of all monotone diffusion-degenerate Nagumo fronts connecting the equilibrium states $u = \alpha$ with $u = 0$ and traveling with speed $c > \bbc(\alpha)$ are spectrally stable in an exponentially weighted energy space $L^2_a = \{ e^{ax} u \in L^2\}$. More precisely, there exists $a > 0$ such that
\[
\sigma(\cL)_{|L^2_a} \subset \{ \lambda \in \C \, : \, \Re \lambda \leq 0\},
\]
where $\cL : L^2_a \to L^2_a$ denotes the linearized operator around the traveling front and $\sigma(\cL)_{|L^2_a}$ denotes its spectrum computed with respect to the energy space $L_a^2$.
\end{theorem}

A few remarks are in order. First, notice that Theorem \ref{mainthmNd} provides no information about a ``spectral gap" (a positive distance between the spectrum and the imaginary axis, except for the eigenvalue zero associated to translations of the wave, of course, which is usually isolated and simple). This is a consequence of the method of proof. Such information would play a key role in the extension of the present spectral analysis to the nonlinear stability problem. Second, it is to be observed that exponentially weighted spaces are needed because degenerate-diffusion equations of Nagumo type underlie a greater variety of traveling waves and there is a whole family of fronts parametrized by their speeds. There exist heteroclinic connections involving the (unstable) middle equilibrium state, $u =\alpha$, in contrast with the standard (strictly parabolic) reaction-diffusion Nagumo equation for which no exponential weights are needed to study the stability of traveling fronts. Finally, we mention at this point an early work by Hosono \cite{Hos86}, which establishes the stability of the sharp traveling front with speed $c = c_*$ for reaction diffusion equations in the ``porous medium'' form (that is, for $D(u) = mu^{m-1}$, with $m \geq 1$) and reaction function $f$ of Nagumo (or bistable type). Hosono's analysis is based on comparison principles and on the construction of super- and sub-solutions to the degenerate parabolic problem. (Notice, however, that this diffusion coefficient does not satisfy assumptions \eqref{hypD} except in the case $m = 2$; in addition, the method of proof is based on particular properties of solutions to the porous medium equation, more precisely, on the maximum principle.) Up to our knowledge, Hosono's paper is the only rigorous work on the stability of diffusion-degenerate Nagumo fronts in the literature. Our contribution lies primarily on the establishment of a new spectral technique in the case of fronts of any speed $c > \bbc(\alpha) > c_*$ with degenerate diffusion, and these ideas may prove useful in the case of systems with some degeneracy attached to them, where the standard techniques for scalar equations could be difficult to apply. 
 
\subsection*{Plan of the paper} This article is organized as follows. In Section \ref{secstructure} we review the existence results due to S\'{a}nchez-Gardu\~{n}o and Maini  \cite{SaMa97}. As a by-product of their analysis we establish the asymptotic decay of the waves and some structural properties that are needed in the stability study. In Section \ref{secspecprob} we pose the spectral stability problem. In particular, we recall Weyl's partition of spectrum and the new partition introduced in \cite{LeP20} to deal with degenerate diffusions. Section \ref{secparreg} contains a generalization to conjugated operators of the parabolic regularization technique and the convergence of operators in the generalized sense that is needed to locate a subset of the essential spectrum. In Section \ref{secenergyest}, we establish energy estimates on the solutions to the spectral equation for conjugated operators under particular hypotheses on the decaying properties of such solutions (see Assumption \ref{assumcrucial} below). The basic energy estimate (see Lemma \ref{lembee}) is the main tool to locate the point spectrum. The final section \S \ref{secNd} is devoted to prove the main Theorem \ref{mainthmNd}. 

\subsection*{On notation} We denote the real and imaginary parts of a complex number $\lambda$ by $\Re\lambda$ and $\Im\lambda$, respectively, as well as complex conjugation by ${\lambda}^*$. We use lowercase boldface roman font to indicate column vectors (e.g., $\bw$), and with the exception of the identity matrix $\Id$, we use upper case boldface roman font to indicate square matrices (e.g., $\A$). Linear operators acting on infinite-dimensional spaces are indicated with calligraphic letters (e.g., $\cL$ and $\cT$), except for the identity operator which is indicated by $\Idd$. In the sequel, $L^2$ and $H^m$, $m \in \Z$, will denote the standard Sobolev spaces of complex-valued functions on the real line, $L^2(\R;\C)$ and $H^m(\R;\C)$, respectively, except when it is explicitly stated otherwise. They are endowed with the standard inner products,
\[
\langle u,v \rangle_{L^2} = \int_\R {u} v^* \, dx, \qquad \langle u,v \rangle_{H^m} = \sum_{k=1}^m \langle \partial_x^k u, \partial_x^k v \rangle_{L^2}.
\]
Let $\ccM(\R;\C)$ denote the set of  Lebesgue measurable complex-valued functions in $\R$. For any $a \in \R$, the exponentially weighted Sobolev spaces, defined as
\[
H^m_a (\R;\C) = \{ v \in \ccM(\R ;\C) \, : \, e^{ax} v \in H^m(\R;\C)\},
\]
for $m \in \Z$, $m\geq 0$, are Hilbert spaces endowed with the inner product and norm,
\[
\langle u,v \rangle_{H^m_a} := \langle e^{ax} u, e^{ax} v \rangle_{H^m}, \qquad \| v \|^2_{H^m_a} := \| e^{ax} v 
\|^2_{H^m} = \langle v,v \rangle_{H^m_a}.
\]
According to custom we denote $L_a^2 = H^0_a$. We use the standard notation in asymptotic analysis \cite{Erde56}, in which the symbol ``$\sim$" means  ``behaves asymptotically like" as $x \to x_*$, more precisely, $f \sim g$ as $x \to x_*$ if $f - g = o(|g|)$ as $x \to x_*$ (or equivalently, $f/g \to 1$ as $x \to x_*$ if both functions are positive).

\section{Structure of monotone diffusion-degenerate Nagumo fronts}
\label{secstructure}

In this section we recall the existence theory of diffusion-degenerate Nagumo fronts and prove some structural properties which will be useful later on.

\subsection{Monotone diffusion-degenerate Nagumo fronts}

The existence of diffu\-sion-degenerate Nagumo fronts for equations of the form \eqref{degRD}, under assumptions \eqref{hypD} and \eqref{bistablef}, was proved by S\'{a}nchez-Gardu\~{n}o and Maini 
\cite{SaMa97}. The authors analyze the local and global phase portraits of the associated ODE system and prove the emergence of heteroclinic connections. Since the Nagumo reaction function has another equilibrium point at $u=\alpha \in (0,1)$, the structure of the fronts is richer than that of the Fisher-KPP case \cite{SaMa94a,SaMa95}. S\'{a}nchez-Gardu\~{n}o and Maini  \cite{SaMa97} show, for instance, the existence of monotone, oscillating and sharp fronts. We specialize our observations, however, to monotone fronts only. 

If we suppose that $u(x,t) = \varphi(x-ct)$ is a traveling wave solution to \eqref{degRD} with speed $c \in \R$, then the profile function $\varphi : \R \to \R$ is a solution to 
\begin{equation}
\label{fronteq}
(D(\varphi) \varphi_\xi)_\xi + c \varphi_\xi + f(\varphi) = 0,
\end{equation}
where $\xi = x - ct$ denotes the translation variable. We write the asymptotic limits of the traveling wave as
\[
u_\pm := \varphi(\pm \infty) = \lim_{\xi \to \pm \infty} \varphi(\xi).
\]
It is assumed that $u_+$ and $u_-$ are equilibrium points of the reaction function, that is, $u_\pm \in \{1,0,\alpha\}$ in the Nagumo case. Written as a first order system, equation \eqref{fronteq} is recast as
\begin{equation}
\label{origsys}
\begin{aligned}
\frac{d \varphi}{d\xi} &= v \\
D(\varphi) \frac{dv}{d\xi} &= -cv -D'(\varphi)v^2 - f(\varphi).
\end{aligned}
\end{equation}
Aronson \cite{Aron80} (see also \cite{SaMa97}) overcomes the singularity at $\varphi = 0$ by introducing the parameter $\tau = \tau(\xi)$, such that
\[
\frac{d\tau}{d \xi} = \frac{1}{D(\varphi(\xi))},
\]
and the system is transformed into
\begin{equation}\label{odeSys}
\begin{aligned}
\frac{d\varphi}{d\tau} &= D(\varphi)v \\
 \frac{dv}{d\tau} &= -cv -D'(\varphi)v^2 - f(\varphi).
\end{aligned}
\end{equation}

Heteroclinic trajectories of both systems are equivalent, so the analysis focuses on the study of the topological properties of equilibria for system \eqref{odeSys}, which depend upon the reaction function $f$. 

Let us suppose that $f = f(u)$ is of Nagumo type, satisfying \eqref{bistablef}. In \cite{SaMa97}, the authors define the function $\mathscr{D}: [0,1] \to \mathbb{R}$ as
\begin{equation}
\label{defcalD}
\mathscr{D}(\varphi) := \int_{0}^{\varphi} D(u) f(u) du.
\end{equation}
Likewise, let us define the following threshold speed
\begin{equation}
\label{defcstar}
\bbc (\alpha) := 2 \sqrt{D(\alpha) f'(\alpha)} > 0.
\end{equation}
The existence of monotone fronts can be summarized as follows.
\begin{proposition}[monotone Nagumo fronts \cite{SaMa97}]
\label{propstructure}
If the function $D = D(u)$ satisfies \eqref{hypD} and $f = f(u)$ is of Nagumo (or bistable) type satisfying 
\eqref{bistablef}, then there exist monotone traveling wave solutions to equation \eqref{degRD} which can be 
classified as follows:
\begin{itemize}
\item[(I)] \emph{Degenerate, monotone increasing fronts:} for each speed value $c > \bbc(\alpha)$, there exists a traveling wave solution $\varphi = \varphi(x-ct)$ with
\[
u_+ = \varphi(+\infty) = \alpha, \qquad u_- = \varphi(-\infty) = 0,
\]
and $\varphi_\xi > 0$ for all $\xi \in \R$. These fronts are diffusion degenerate at $u_- = 0$ as 
$\xi \to -\infty$.
\item[(II)] \emph{Non-degenerate, monotone decreasing fronts:} for each speed value $c > \bbc(\alpha)$, there exists a traveling wave solution $\varphi = \varphi(x-ct)$ with
\[
u_+ = \varphi(+\infty) = \alpha, \qquad u_- = \varphi(-\infty) = 1,
\]
and $\varphi_\xi < 0$ for all $\xi \in \R$. These fronts are non-degenerate in the sense that $D(u_\pm) > 0$ and 
$D(\varphi) \geq \delta_0 > 0$ for all values of $\varphi$.
\item[(III)] \emph{Stationary degenerate fronts:} when $c = 0$ 
and $\mathscr{D}(1) = 0$ there exists a unique monotone increasing front,$\varphi = \varphi(x)$, with 
\[
u_+ = \varphi(+\infty) = 1, \qquad u_- = \varphi(-\infty) = 0,
\]
as well as a unique monotone decreasing front, $\varphi = \varphi(x)$, with 
\[ 
u_+ = \varphi(+\infty) = 0, \qquad u_- = \varphi(-\infty) = 1.
\]
These fronts are diffusion-degenerate in the sense that $D$ vanishes at one of the end points.
\end{itemize} 
\end{proposition}

\begin{remark}
The main Theorem 1 in \cite{SaMa97} also establishes the existence of sharp traveling waves, as well as that of pulses and oscillatory (non-monotone) fronts, which are not considered in the present analysis. 
\end{remark}

As it was mentioned in the introduction, in this paper we focus on the stability properties of monotone degenerate fronts of type I, as described in Proposition \ref{propstructure}. An illustration of the degenerate monotone fronts of type I can be observed in Figure \ref{figfronts}. These fronts connect $u = 0$ with $u = \alpha$ in their limits as $\xi \to -\infty$ and $\xi \to +\infty$, respectively. This behavior is due to the fact that, for each speed value $c > \overline{c}(\alpha)$, the point $P_0 = (0,0)$ is a saddle-node and $P_\alpha = (\alpha, 0)$ is a locally stable node as equilibria of the first order system \eqref{origsys}; moreover, each front is monotone increasing, with $\varphi_\xi > 0$ for all $\xi \in \R$ (see S\'{a}nchez-Gardu\~{n}o and Maini \cite{SaMa97} for more details).

\begin{figure}[t]
\begin{center}
\includegraphics[scale=.55, clip=true]{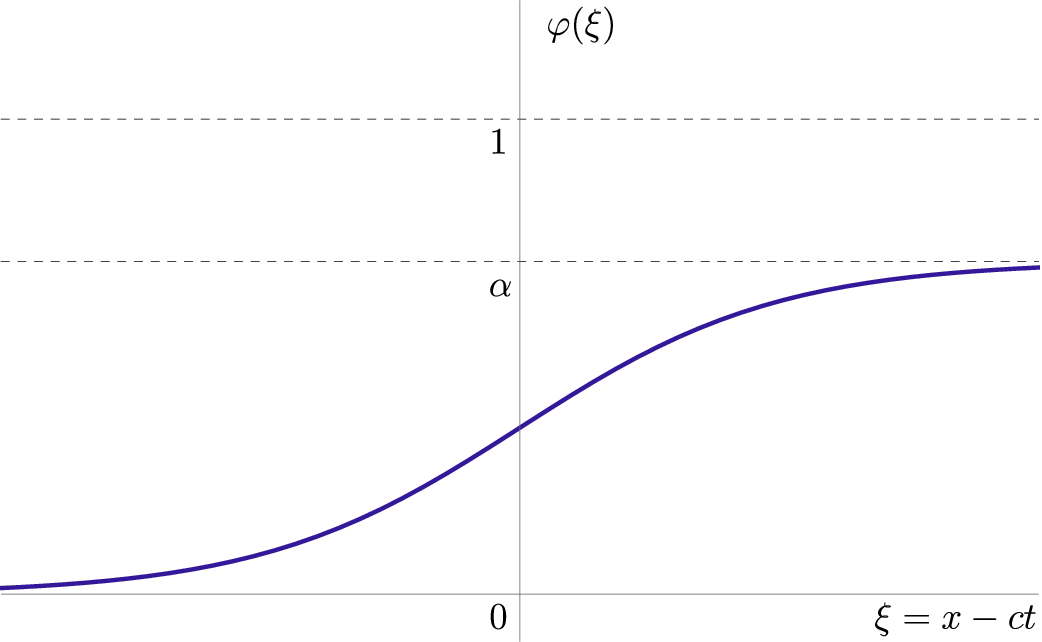}
\end{center}
\caption{Illustration of the family of monotone Nagumo traveling fronts of type I from Proposition \ref{propstructure}. For each value of $c > \bbc(\alpha)$, the front (in blue color) is monotone increasing and connects $u=0$ with $u=\alpha$; it is degenerate in the sense that $D$ vanishes at one of the end points, namely, at $u = 0$ (color online).}\label{figfronts}
\end{figure}

\subsection{Asymptotic decay}

In this section we examine the asymptotic behavior of the traveling fronts as $\xi \to \pm \infty$, which can be deduced from the existence analysis of \cite{SaMa97}. We shall use this information in the course of the stability study.

\begin{lemma}[asymptotic decay]\label{lemdecayNd}
Let $\varphi = \varphi(\xi)$ be a monotone increasing degenerate Nagumo front, for which $u_+ = \alpha$, $u_- = 0$, $\varphi_\xi > 0$, traveling with speed $c > \bbc(\alpha)$ (front of type I in Proposition \ref{propstructure}). Then $\varphi$ behaves asymptotically as 
\begin{equation}
\label{decayNdn}
|\partial_\xi^j(\varphi - u_+)| = |\partial_\xi^j(\varphi - \alpha)| = O(e^{-\eta \xi}), \quad \text{as } \; \xi \to +\infty, \; j=0,1,
\end{equation}
for $\eta = (2D(\alpha))^{-1} (c + \sqrt{c^2 - 4D(\alpha)f'(\alpha)}) > 0$; and as,
\begin{equation}
\label{decayNdd}
|\partial_\xi^j(\varphi - u_-)| = |\partial_\xi^j \varphi | = O(e^{|f'(0)|\xi/c}), \quad \text{as } \; \xi \to -\infty, \; j=0,1,
\end{equation}
on the degenerate side. Moreover, $\varphi$ and $\varphi_\xi$ behave asymptotically like
\begin{equation}
\label{asintotaNd}
\varphi \sim C_0 e^{-\eta \xi}, \quad \varphi_\xi \sim C_1 e^{-\eta \xi},
\end{equation}
when $\xi \to +\infty$, and like
\begin{equation}
\label{asintotaNd2}
\varphi \sim C_2 e^{|f'(0)| \xi/c}, \quad \varphi_\xi \sim C_3 e^{|f'(0)| \xi/c},
\end{equation}
when $\xi \to -\infty$, for appropriate positive uniform constants $C_j > 0$.
\end{lemma}
\begin{proof}
Let us first examine the asymptotic behavior on the non-degenerate side as $\xi \to +\infty$. The linearization of system \eqref{origsys} around the equilibrium point $P_\alpha = (\alpha,0)$ is given by
\[
(D_{(\varphi,v)}\mathbf{g})_{|(\alpha,0)} = \begin{pmatrix}
0 & 1 \\ -f'(\alpha)/D(\alpha) & -c/D(\alpha)
\end{pmatrix},
\]
having only one negative eigenvalue,
\[
\mu = - \frac{1}{2D(\alpha)} \big( c + \sqrt{c^2 - \bbc(\alpha)^2}\big) =: - \eta,
\]
with $\eta > 0$. By standard ODE estimates around a hyperbolic end point the trajectory behaves asymptotically as
\[
\varphi \sim C_0 e^{-\eta \xi}, \quad \varphi_\xi \sim C_1 e^{-\eta \xi},
\]
when $\xi \to +\infty$ for some uniform constants $C_0 , C_1 > 0$.

To verify the exponential decay on the degenerate side, notice that $P_0 = (0,0)$ is a non-hyperbolic point for system \eqref{odeSys} for all admissible values of the speed $c \geq \bbc(\alpha) = \sqrt{4D(\alpha)f'(\alpha)} > 0$, and we need higher order terms to approximate the trajectory along a center manifold. Let us denote the former as $v = h(\varphi)$. After an application of the center manifold theorem, we find that $P_0$ is locally a saddle-node, and the center manifold has the form
\[
h(\varphi) = - \frac{f'(0)}{c} \varphi - \frac{1}{2c^3}\big( f''(0) c^2 + 4 D'(0) f'(0)^2 \big) \varphi^2 + O(\varphi^3),
\]
as $\varphi \to 0^+$ (see S\'{a}nchez-Gardu\~{n}o and Maini \cite{SaMa97} for details). The trajectory leaves the saddle-node along the center manifold for $\varphi \to 0^+$. Therefore, for $\xi \to -\infty$, the trajectory behaves as
\[
\varphi_\xi = h(\varphi) \sim - \frac{f'(0)}{c} \varphi \geq 0,
\]
as $\varphi \to 0^+$, yielding
\[
\varphi = O(e^{|f'(0)|\xi/c}), \quad \text{as } \; \xi \to -\infty.
\]
This proves the result.
\end{proof}

\subsection{Further properties}

We finish this section by proving that the fronts under consideration have the required regularity for the spectral analysis that follows. For later use, we also prove an auxiliary lemma that guarantees the boundedness of a particular coefficient. 

\begin{lemma}[$C^4$-regularity]
\label{lemC4}
Under assumptions \eqref{hypD} and \eqref{bistablef}, all  the monotone Nagumo fronts of Proposition \ref{propstructure} are of class $C^4(\R)$.
\end{lemma}
\begin{proof}
Since $D \in C^4$, $f \in C^3$ and $(\varphi,v) = (\varphi, \varphi_\xi)$ is a solution to a nonlinear autonomous system of the form
\[
\frac{d}{d\xi} \begin{pmatrix} \varphi \\ \varphi_\xi \end{pmatrix} = \mathbf{g}(\varphi, \varphi_\xi),
\]
with $\mathbf{g} = \mathbf{g}(\varphi,v)$ of class $C^3$ in $(\varphi,v)$, then $\varphi$ is at least of class $C^4$ for $\xi \in \R$. 
\end{proof}

\begin{lemma}[$H^2$-regularity]
\label{lemH2}
All the monotone degenerate Nagumo fronts of type I in Proposition \ref{propstructure} satisfy $\varphi_\xi \in H^2$.
\end{lemma}
\begin{proof}
Since each profile function $\varphi$ is at least of class $C^4$ as a function of $\xi \in \R$, in order to show that $\varphi_\xi \in H^2(\R)$ it suffices to verify that $\varphi_\xi$ and its derivatives decay sufficiently fast as $\xi \to \pm \infty$. For the degenerate fronts of type I under consideration, this behaviour follows directly from exponential decay (see \eqref{decayNdn} and \eqref{decayNdd}), which yields $\varphi_\xi \in L^2(\R)$. The same conclusion can be obtained for $\varphi_{\xi \xi}$ and $\varphi_{\xi \xi \xi}$ by a bootstrapping argument. We omit the details.
\end{proof}

\begin{lemma}
\label{lemauxi}
For all monotone degenerate Nagumo fronts of type I, there exists a uniform constant $C > 0$ such that
\[
\sup_{\xi \in \R} \left| D(\varphi) \frac{\varphi_{\xi \xi}}{\varphi_{\xi}}\right| \leq C.
\]
\end{lemma}
\begin{proof}
On the degenerate side, as $\xi \to -\infty$, one has $\varphi \to 0^+$ and $D(\varphi) = D'(0) \varphi + O(\varphi^2)$. Therefore, from exponential decay of $\varphi_\xi$ and $\varphi_{\xi \xi}$ it is easy to verify the asymptotic behaviour
\[
D(\varphi) \frac{\varphi_{\xi\xi}}{\varphi_\xi} \sim C D'(0) e^{|f'(0)|\xi/c} \to 0,
\]
as $\xi \to -\infty$, for some uniform constant $C > 0$. On the non-degenerate side, since $D(\varphi) \geq \delta > 0$ for $\varphi \to \alpha$ as $\xi \to +\infty$, we conclude directly that $D(\varphi) \varphi_{\xi\xi}/\varphi_\xi \to 0$ as $\xi \to +\infty$. Finally, since the coefficient $D(\varphi) \varphi_{\xi\xi}/\varphi_\xi$ is continuous in $\xi \in [-R,R]$ for any $R> 0$ and vanishes as $\xi \to \pm \infty$ the conclusion follows.
\end{proof}

\section{The spectral stability problem}
\label{secspecprob}

\subsection{Perturbation equations}

Suppose that $u(x,t) = \varphi(x-ct)$ is one of the monotone diffusion-degenerate Nagumo traveling fronts of Proposition \ref{propstructure}, traveling with speed $c \in \R$. In all the cases under consideration, the speed is non-negative, so we assume from this point on that $c \geq 0$. With a slight abuse of notation we make the change of variables 
$x \to x-ct$, where now $x$ denotes the translation variable. We shall keep this notation for 
the rest of the paper. In the new coordinates, equation \eqref{degRD} is recast as
\begin{equation}
\label{newRD}
u_t = ( D(u) u_x)_x + cu_x + f(u),
\end{equation}
for which traveling fronts are stationary solutions, $u(x,t) = \varphi(x)$, satisfying the profile equation,
\begin{equation}\label{profileq}
(D(\varphi)\varphi_x)_x+c\varphi_x+f(\varphi)=0.
\end{equation}
The front connects asymptotic equilibrium points of the reaction: $\varphi(x) \to u_{\pm}$ as $x 
\to \pm \infty$, where $u_{\pm} \in \{ 0,1,\alpha\}$, depending on the case (see Proposition \ref{propstructure}). The wave can be either monotone increasing, $\varphi_x>0$, or monotone decreasing, $\varphi_x<0$.

Let us consider solutions to \eqref{newRD} of the form $\varphi(x)+u(x,t)$, where, from now on, $u$ denotes a 
perturbation. Substituting we obtain the nonlinear perturbation equation
\[  
u_t = (D(\varphi+u)(\varphi+u)_x)_x + cu_x+ c\varphi_x +f(u+\varphi). 
\]
Linearizing around the front and using the profile equation \eqref{profileq} we get
\begin{equation}\label{lineareq}
u_t = (D(\varphi)u)_{xx} +cu_x+f'(\varphi)u.
\end{equation}
The right hand side of equation \eqref{lineareq}, regarded as a linear operator acting on an appropriate 
Banach space $X$, naturally defines the following spectral problem 
\begin{equation}\label{spectralp}
\lambda u = \mathcal{L}u,
\end{equation}
where $\lambda \in \mathbb{C}$ is the spectral parameter, and
\begin{equation}\label{opL}
\begin{aligned}
\mathcal{L}&: \mathcal{D}(\mathcal{L}) \subset X \to X, \\
\mathcal{L}u &=(D(\varphi)u)_{xx} +cu_x+f'(\varphi)u.
\end{aligned}
\end{equation}

Loosely speaking, by spectral stability we understand the absence of solutions $u \in X$ to equation \eqref{spectralp} for $\Re \lambda > 0$, that is, the situation in which the spectrum of the operator $\cL$ belongs to the stable half plane with $\Re \lambda \leq 0$. Clearly, the spectrum of the operator depends upon the choice of the Banach space $X$. In the stability analysis of nonlinear waves, it is customary to consider energy spaces $X = L^{2}$ and $\mathcal{D}(\mathcal{L})= H^{2}$, 
so that $\mathcal{L}$ can be regarded as a closed, densely defined operator acting on $L^{2}$ (see Remark \ref{remclosed} below). In this fashion, the stability analysis of the operator $\cL$ pertains to \textit{localized 
perturbations}.

\subsection{Partitions of spectra}

In this section we define the resolvent and spectra of a linearized operator and what we understand as spectral stability. More importantly, we introduce two different partitions of the spectrum, which will be used in the stability analysis.

Let $X$ and $Y$ be Banach spaces, 
and let $\ccC(X,Y)$ and $\ccB(X,Y)$ denote the sets of all closed and bounded linear operators from $X$ to 
$Y$, respectively. For any $\cL \in \ccC(X,Y)$ we denote its domain as $\cD(\cL) \subseteq X$ and its range 
as $\cR(\cL) := \cL (\cD(\cL)) \subseteq Y$. We say $\cL$ is densely defined if $\overline{\cD(\cL)} = X$.

\begin{definition}[resolvent and spectrum]
 \label{defspec}
Let $\cL \in \ccC(X,Y)$ be a closed, densely defined operator. Its \textit{resolvent} 
$\rho(\cL)$ is defined as the set of all complex numbers $\lambda \in \C$ such that $\cL - \lambda$ is 
injective, $\cR(\cL - \lambda) = Y$ and $(\cL - \lambda)^{-1}$ is bounded. Its \textit{spectrum} is defined 
as $\sigma(\cL) := \C \backslash \rho(\cL)$. 
\end{definition}

\begin{definition}[$X$-spectral stability]
\label{defspecstab}
 We say the traveling front $\varphi$ is $X$-\textit{spectrally stable} if 
 \[ 
 \sigma(\mathcal{L}) \subset \{ \lambda \in \C \, : \, \Re{\lambda} \leq 0\}.
  \]
\end{definition}

In the analysis of stability of nonlinear waves (cf. \cite{KaPro13, San02}) the spectrum is often partitioned into essential
and isolated point spectrum. This definition is originally due to Weyl \cite{We10} (see also \cite{EE87,Kat80}).

\begin{definition}[Weyl's partition of spectrum]
\label{defsigmaone}
Let $\cL \in \ccC(X,Y)$ be a closed, dense\-ly defined linear operator.
We define its \textit{isolated point spectrum} and its \textit{essential spectrum}, as the sets
\[
\begin{aligned}
\iptsp(\cL) &:= \{ \lambda \in \C \, : \, \cL - \lambda \,\text{ is Fredholm with index zero and
non-trivial kernel} \}, \, \text{and,} &\\
\ess(\cL) &:= \{ \lambda \in \C \,: \, \cL - \lambda \,\text{ is either not Fredholm, or has index different 
from zero} \},
\end{aligned}
\]
respectively. 
\end{definition}

We remind the reader that an operator $\mathcal{L} \in \ccC(X,Y)$ is said to be Fredholm if its 
range $\mathcal{R(L)}$ is closed, and both its nullity, $\nul\mathcal{L} = \dim \ker \mathcal{L}$, and 
its deficiency, $\mathrm{def} \,\mathcal{L} = \mathrm{codim} \, \mathcal{R(L)}$, are finite. $\cL$ is said to 
be semi-Fredholm if $\cR(\cL)$ is closed and at least one of $\nul \cL$ and $\mathrm{def} \, \cL$ is finite. 
In both cases the index of $\cL$ is defined as
$\ind \mathcal{L} = \nul \mathcal{L} - \mathrm{def} \, \mathcal{L}$ (cf. \cite{Kat80}).

\begin{remark} 
Notice that since $\cL$ is a closed operator, then $\sigma(\cL) = \iptsp(\cL) \cup \ess(\cL)$ (cf. Kato \cite{Kat80}, p. 167). There are many definitions of essential spectrum in the literature. Weyl's definition makes it easy to compute and has the advantage that the remaining point spectrum, $\iptsp$, is a discrete set of isolated eigenvalues (see Remark 2.2.4 in \cite{KaPro13}). 
\end{remark}

In the present context with degenerate diffusion, however, this partition is not particularly useful when analyzing fronts that connect a degenerate point for the diffusion due to the loss of hyperbolicity of the asymptotic coefficients of the operator when the spectral problem is written as a first order system. Consequently, in this paper we shall also employ the following partition of the spectrum (see also \cite{LeP20}).

\begin{definition}
 \label{defspecd}
Let $\cL \in \ccC(X,Y)$ be a closed, densely defined operator. We define the following subsets of the complex 
plane:
\[
 \begin{aligned}
 \ptsp(\cL) := &\{ \lambda \in \C \, : \, \cL - \lambda \; \text{is not injective}\},\\ 
 \spd(\cL) := &\{ \lambda \in \C \, : \, \cL - \lambda \; \text{ is injective, } \cR(\cL - \lambda) \text{ is 
closed and } \cR(\cL- \lambda) \neq Y\},\\
\sppi(\cL) := &\{ \lambda \in \C \, : \, \cL - \lambda \; \text{ is injective and } \cR(\cL - \lambda) 
\text{ is not closed}\}.
 \end{aligned}
\]
\end{definition}

\begin{remark}
Observe that the sets $\ptsp(\cL)$, $\sppi(\cL)$ and $\spd(\cL)$ are clearly disjoint and, since $\cL$ 
is closed, we know that if $\cL$ is invertible then $\cL^{-1} \in \ccB(Y,X)$. Consequently,
\[
 \sigma(\cL) = \ptsp(\cL) \cup \sppi(\cL) \cup \spd(\cL).
\] 
Also note that the set of isolated eigenvalues with finite multiplicity (see Definition \ref{defsigmaone}) is contained in $\ptsp$, that is, $\iptsp(\cL) \subset \ptsp(\cL)$, and that $\lambda \in \ptsp(\cL)$ if and only if there exists $u \in \cD(\cL)$, $u \neq 0$ such that $\cL u = \lambda u$.
\end{remark}

\begin{remark}\label{remponla}
Note that the set $\spd(\cL)$ is clearly contained in what is known as the \textit{compression spectrum} 
(see, e.g., \cite{Jerib15}, p. 196):
\[
 \spd(\cL) \subset {\sigma_\mathrm{\tiny{com}}}(\cL) := \{ \lambda \in \C \, : \, \cL - \lambda \; \text{ is 
injective, and } \overline{\cR(\cL-\lambda)} \neq Y\},
\]
whereas the set $\sppi(\cL)$ is contained in the \textit{approximate spectrum} \cite{EE87}, defined as
\[
\begin{aligned}
 \sppi(\cL) \subset {\sigma_\mathrm{\tiny{app}}}(\cL) := &\{ \lambda \in \C \, : \, \text{there exists $u_n \in \cD(\cL)$ with $\|u_n\| = 1$} \\ & \; \text{ such that $(\cL - \lambda)u_n \to 0$ in 
$Y$ as $n \to +\infty$}\}.
\end{aligned}
\]
The last inclusion follows from the fact that, for any $\lambda \in \sppi(\cL)$, the range of $\cL - \lambda$ is 
not closed and, therefore, there exists a \textit{singular sequence}, $u_n \in \cD(\cL)$, $\|u_n\| =1$ such 
that $(\cL - \lambda)u_n \to 0$, which contains no convergent subsequence (see Theorems 5.10 and 5.11 in Kato \cite{Kat80}, p. 233). Moreover, since the space $L^2$ is reflexive, this sequence can be chosen so that $u_n \rightharpoonup 0$ in $L^2$ (see \cite{EE87}, p. 415). On the other hand, it is clear that $\ptsp(\cL) \subset {\sigma_\mathrm{\tiny{app}}}(\cL)$, as well.
\end{remark}

In the sequel, $\sigma(\cL)_{|X}$ will denote the $X$-spectrum of the linearized operator \eqref{opL} with dense domain $\cD(\cL) \subset X$. For example, when it is computed with respect to a local energy weighted space $L_a^2$ for some $a\in \R$, its spectrum is denoted by $\sigma(\cL)_{|L_a^2}$. This notation applies to any subset of the spectrum.

\begin{remark}[translation eigenvalue]
\label{remef0}
We recall that $\lambda = 0$ always belongs to the $L^2$-point spectrum because
\[ 
\mathcal{L} \varphi_x = \partial_x \big( (D(\varphi)\varphi_x)_x+c\varphi_x+f(\varphi) \big)=0,
\]
in view of the profile equation \eqref{profileq} and the fact that $\varphi_x \in 
H^{2} = \cD(\cL)$ (see Lemma \ref{lemH2}). Thus, $\varphi_x$ is the eigenfunction associated to the eigenvalue 
$\lambda=0$.
\end{remark}

\subsection{Differential operators on the real line}
\label{secdiffop}

For convenience of the reader, in this section we gather some known facts of the spectral theory of linear differential operators. For more information the reader is referred to \cite{EE87,Nai-LDO2}.

Consider a linear differential operator of the form
\begin{equation}
\label{gendiffop}
\begin{aligned}
{\cL} &: \cD({\cL}) \subset L^2(\R) \to L^2(\R),\\
\cL  &= b_m(x) \partial_x^m + \ldots + b_1(x) \partial_x + b_0 (x) \Idd,
\end{aligned}
\end{equation}
where $m \geq 1$ denotes the order of the operator and $x \in \R$ is the spatial variable. It is assumed that the coefficients are real, smooth enough and uniformly bounded, that they have finite asymptotic limits at $x=\pm \infty$ and that $b_m$ is non-negative; more precisely, that
\begin{equation}
\label{hypbjs}
\begin{aligned}
b_j \in C^m(\R) \cap L^\infty(\R), & & 0 \leq j \leq m,\\
\partial_x^k b_j \in L^\infty(\R), & & 1 \leq j,k \leq m,\\
b_j^\pm := \lim_{x \to \pm \infty} b_j(x) \in \R,  & & 0 \leq j \leq m,\\
\text{and } \quad b_m(x) \geq 0, & &\text{for all } x \in \R.
\end{aligned}
\end{equation}
The domain $\cD(\cL)$ is supposed to be a dense subspace of $L^2(\R)$ and, consequently, $\cL : L^2 \to L^2$ is densely defined. For example, let us consider
\[
\cD(\cL) = H^m(\R),
\]
which is clearly dense in $L^2(\R)$.

We remind the reader that if $\cL$ and $\cT$ are linear operators from $X$ to $Y$ (Banach spaces) such that $\cD(\cL) \subset \cD(\cT)$ and $\cL u = \cT u$ for all $u \in \cD(\cL)$, then $\cT$ is called an \emph{extension of} $\cL$ (denoted $\cT \supset \cL$) and $\cL$ is called a \emph{restriction of} $\cT$ (denoted $\cL \subset \cT$); cf. \cite{EE87,Kat80}. An operator $\cT$ from $X$ to $Y$ is \emph{closable} if $\cT$ has a closed extension. $\overline{\cT}$ is the smallest closed extension of $\cT$ if any other closed extension of $\cT$ is also an extension of $\overline{\cT}$.
\begin{proposition}
\label{propclosable}
Under assumptions \eqref{hypbjs}, $\cL : \cD(\cL) = H^m(\R) \subset L^2(\R) \to L^2(\R)$ given in \eqref{gendiffop} is a closable operator.
\end{proposition}
\begin{proof}
Since $L^2$ is a reflexive Hilbert space we identify $L^2$ with its dual. Thus, the formal adjoint of $\cL$, given by
\[
\begin{aligned}
\cL^* &: \cD(\cL^*) \subset L^2(\R) \to L^2(\R),\\
\cL^* v &= \sum_{k=0}^m (-1)^k \partial_x^k (b_k(x) \overline{v}(x)),
\end{aligned}
\]
has domain (cf. \cite{Kat80}, p. 167)
\[
\cD(\cL^*) = \{ v \in L^2(\R) \, : \, \exists \, f \in L^2 \; \text{such that } \, \langle v, \cL u \rangle_{L^2} = \langle f, u \rangle_{L^2}, \; \text{for all } u \in H^m\}.
\]

It suffices to show that $H^m \subset \cD(\cL^*)$ to conclude that $\cD(\cL^*)$ is dense in $L^2$. Let $v \in H^m$. Then for any $u \in \cD(\cL) = H^m$ we have, after integration by parts, that
\[
\langle v, \cL u \rangle_{L^2} = \langle \cL^* v, u \rangle_{L^2}.
\]
Since clearly $f := \cL^* v \in L^2$ whenever $v \in H^m$, we obtain that $v \in \cD(\cL^*)$. Thus, $\cD(\cL^*)$ is dense in $L^2$.

Now, since $\cL$ is densely defined, $\cL^*$ coincides with the proper adjoint of $\cL$ (see Kato \cite{Kat80}, p. 167). As $\cL^*$ is densely defined, one applies Theorem 5.28 in \cite{Kat80} (p. 168) to conclude that $\cL$ is closable. Moreover, since $L^2$ is a reflexive space, by Theorem 5.29 in \cite{Kat80} we also have that $\cL^*$ is closed and $\cL^{**} = \overline{\cL} \supset \cL$, where $\overline{\cL}$ denotes the smallest closed extension of $\cL$.
\end{proof}

If we further assume that the highest order coefficient is uniformly bounded above and below, that is,
\begin{equation}
\label{bmbded}
0 < C_0 \leq b_m(x) \leq C_1, \qquad x \in \R, \; \; \text{uniform} \; C_j > 0,
\end{equation}
then we have the following
\begin{proposition}
\label{propclosed}
Under hypotheses \eqref{hypbjs} and \eqref{bmbded}, $\cL : \cD(\cL) = H^m(\R) \subset L^2(\R) \to L^2(\R)$ given in \eqref{gendiffop} is a closed operator. 
\end{proposition}
\begin{proof}
Define $a_k(x):= b_k(x) b_m(x)^{-1}$, for $k = 0, \ldots, m-1$. Then from \eqref{hypbjs} and \eqref{bmbded} it is clear that $a_k \in W^{1,\infty}(\R)$. Apply Lemma 3.1.2 in reference \cite{KaPro13} (p. 40) to deduce that the operator $\cA : H^m \subset L^2 \to L^2$, $\cA := \partial_x^m + a_{m-1}(x) \partial_x^{m-1} + \ldots + a_0 \Idd$ is closed. Now take any sequence $u_j \in H^m$ such that $u_j \to u$ and $\cL u_j \to v$ in $L^2$. From \eqref{bmbded} we have that $\widetilde{v} = b_m(x)^{-1}v \in L^2$ and that
\[
\| \cA u_j - \widetilde{v} \|_{L^2} \leq C \| \cL u_j - v \|_{L^2} \to 0,
\]
as $j \to \infty$. Since $\cA$ is closed we conclude that $u \in H^m$ and $\cA u = \widetilde{v}$, which is tantamount to $\cL u = v$. Thus, $\cL$ is closed.
\end{proof}

Being the highest order coefficient $b_m$ bounded from above and from below, multiplication by $b_m^{-1}$ is a diffeomorphism that does not affect the Fredholm properties of an operator.

\begin{lemma}
\label{lemJFred}
Under assumptions \eqref{hypbjs} and \eqref{bmbded}, $\cL$ is a Fredholm operator with index $\nu$ if and only if $\cJ := b_m(x)^{-1} \cL : H^m \subset L^2 \to L^2$ is Fredholm with index $\nu$.
\end{lemma}
\begin{proof}
Suppose $\cL$ is Fredholm with index $\nu$. Take a sequence $v_j \in \cR(\cJ)$ such that $v_j \to v$ in $L^2$. Then there exists $u_j \in H^m$ such that $\cJ u_j = b_m(x)^{-1} \cL u_j = v_j$, that is, $b_m(x) v_j \in \cR(\cL)$. Thanks to \eqref{bmbded} we have that $v_j \to v$ in $L^2$ if and only if $b_m(x) v_j \to b_m(x) v$ in $L^2$. Since $\cR(\cL)$ is closed we obtain $b_m(x) v \in \cR(\cL)$ or, equivalently, $v \in \cR(\cJ)$. Thus, $\cR(\cJ)$ is closed. Moreover, the transformation $\Phi : \cR(\cJ) \to \cR(\cL)$ defined by $\Phi v = b_m(x) v$ is injective and onto, yielding $\codim \cR(\cJ) = \codim \cR(\cL)$. Finally, observe that if $v \in \ker \cJ$ then $0 = \| b_m(x)^{-1} \cL v \|_{L^2} \geq C_1^{-1}\| \cL v\|_{L^2}$, yielding $v \in \ker \cL$. Likewise, $\ker \cL \subset \ker \cJ$. Hence $\nul \cL = \nul \cJ$ and $\cJ$ is Fredholm with index $\nu = \ind \cL$.
\end{proof}

\begin{remark}
\label{remclosed}
In the present study of the spectral properties of linearized operators around diffusion-degenerate traveling fronts, the highest order coefficient is not uniformly bounded below. Nonetheless, we work with the smallest closed extension of the linearized operator and, without loss of generality, we regard it as a closed, densely defined operator in a suitable energy space.
\end{remark}

\section{Generalized convergence of parabolic regularizations}
\label{secparreg}

In this section we extend the technique introduced in \cite{LeP20} (see Section 5 in that reference) to locate the subset $\sigma_\delta$ of the spectrum, to the case of exponentially weighted Sobolev spaces. The method consists of performing a parabolic regularization of the linearized operator around the wave and on establishing the convergence (in the generalized sense) of the family of regularized operators to the degenerate operator as the regularization parameter tends to zero.

\subsection{The regularized operator on exponentially weighted energy spaces}
\label{secregop}

Let $\varphi = \varphi(x)$ be any of the monotone Nagumo fronts of Proposition \ref{propstructure} and consider the following regularization of the diffusion coefficient: for any $\epsilon > 0$, let us define
\begin{equation}
\label{regD}
D^\epsilon(\varphi) := D(\varphi) + \epsilon.
\end{equation}
Note that, for fixed $\epsilon > 0$, $D^\epsilon(\varphi)$ is positive and uniformly bounded below, including the asymptotic limits, $D^\epsilon(u_\pm) = D(u_\pm) + \epsilon > 0$. 

Let us consider now the following family of linearized, regularized operators defined on weighted energy spaces,
\begin{equation}
\label{perturbOpL}
\begin{aligned}
\cL^\epsilon u &:= (D^\epsilon(\varphi) u)_{xx} + cu_x + f'(\varphi)u,\\
\cL^\epsilon &: \cD_a(\cL^\epsilon) \subset L^2_a \to L^2_a,
\end{aligned}
\end{equation}
for some $a \in \R$ and any $\epsilon \geq 0$. Its domain $\cD_a(\cL^\epsilon)$ is a dense subset of $L^2_a$. For instance, we can take
\[
\cD_a(\cL^\epsilon) := H^2_a.
\]
In this fashion, for any $\epsilon > 0$, $\cL^\epsilon$ is a densely defined, closed operator in $L^2_a$, with domain $\cD_a(\cL^\epsilon) = H^2_a$. For $\epsilon = 0$, $\cL^0 = \cL$ is a densely defined, \emph{closable} operator in $L^2_a$. In this case we work with the closed extension of $\cL$ and, without loss of generality, we regard $\cL^\epsilon : L^2_a \to L^2_a$ as a closed, densely defined operator \emph{for any} $\epsilon \geq 0$ with dense domain $\cD_a(\cL^\epsilon) = H^2_a$ (see Section \ref{secdiffop}).

We are interested in computing the spectrum of $\cL^\epsilon$ with respect to the weighted energy space $L^2_a$. It is well known that the location of the essential spectrum depends upon the weight function under consideration (see, e.g., Kapitula and Promislow \cite{KaPro13}, chapter 3), whereas the point spectrum is unmoved (cf. \cite{KaPro13}, or Proposition 3 in \cite{LeP20}). It is also known \cite{KaPro13} that the spectrum of $\cL^\epsilon$ on $L^2_a$ is equivalent to that of the conjugated operator $\cL^\epsilon_a$ acting on the original $L^2$ space, where
\[
\cL_a^\epsilon := e^{ax} \cL^\epsilon e^{-ax} : \cD(\cL_a^\epsilon) \subset L^2 \to L^2,
\]
\[
\cD(\cL_a^\epsilon) := \{ u \in \ccM(\R;\C) \, : \, e^{-ax} u \in H^2_a \} \, \subset \, L^2.
\]
From its definition it is clear that $\cD(\cL_a^\epsilon) = H^2(\R) \subset L^2(\R)$ and, therefore, $\cL_a^\epsilon$ is a densely defined operator in $L^2$. Once again, if $\epsilon > 0$ then $\cL_a^\epsilon$ is a closed operator in $L^2$, whereas if $\epsilon =0$ the operator is closable and we work with the closed extension of $\cL_a^0$ on $L^2$. Hence, without loss of generality, we regard $\cL_a^\epsilon : \cD(\cL_a^\epsilon) = H^2 \subset L^2 \to L^2$, as a family of closed, densely defined operators for each $\epsilon \geq 0$.

If $\varphi$ is any of the monotone traveling fronts in Proposition \ref{propstructure} and $\cL^\epsilon$ is the linear operator defined in \eqref{perturbOpL}, then after straightforward calculations we find that the associated conjugated operator $\cL_a^\epsilon$ is given by
\begin{equation}
\label{conjop}
\begin{aligned}
\cL^\epsilon_a &: \cD(\cL_a) = H^2 \subset L^2 \to L^2,\\
\cL^\epsilon_a &= D^\epsilon(\varphi) \partial_x^2 + b^\epsilon_{1,a}(x) \partial_x + b^\epsilon_{0,a}(x) \Idd,
\end{aligned}
\end{equation}
with coefficients
\begin{equation}
\label{defbepa}
\begin{aligned}
b^\epsilon_{1,a}(x) &:= 2D^\epsilon(\varphi)_x + c - 2a D^\epsilon(\varphi),\\
b^\epsilon_{0,a}(x) &:= a^2D^\epsilon(\varphi) - 2aD^\epsilon(\varphi)_x -ac + D^\epsilon(\varphi)_{xx} + f'(\varphi).
\end{aligned}
\end{equation}

Set $\epsilon = 0$ in \eqref{conjop} and \eqref{defbepa} to obtain the expression of the conjugated operator of the original (non-regularized) operator $\cL$,
\begin{equation}
\label{conjopL}
\begin{aligned}
\cL_a &: \cD(\cL_a) = H^2 \subset L^2 \to L^2,\\
\cL_a &= D(\varphi) \partial_x^2 + b_{1,a}(x) \partial_x + b_{0,a}(x) \Idd,
\end{aligned}
\end{equation}
where
\begin{equation}
\label{defbpa}
\begin{aligned}
b_{1,a}(x) &:= 2D(\varphi)_x + c - 2a D(\varphi),\\
b_{0,a}(x) &:= a^2D(\varphi) - 2aD(\varphi)_x -ac + D(\varphi)_{xx} + f'(\varphi).
\end{aligned}
\end{equation}
Observe that $\cL_a$ and $\cL_a^\epsilon$ have the same domain, $\cD = H^2$.

Notice as well that for every $\epsilon > 0$, $- \, \cL_\omega^\epsilon$ is a
\textit{strongly elliptic} operator acting on $L^2$ and that multiplication by $D^\epsilon (\varphi)^{-1}$ is 
an isomorphism (recall that $0 < \epsilon \leq D^\epsilon(\varphi) \leq C_\epsilon$ for all $x \in \R$ and some uniform $C_\epsilon > 0$). Henceforth, the Fredholm properties of $\mathcal{L}^\epsilon_a - \lambda$ and those of the operator 
$\mathcal{J}^\epsilon_a(\lambda) : L^2 \to L^2$, densely defined as
\begin{equation}
\label{opJl}
\mathcal{J}^\epsilon_a(\lambda) := D^\epsilon(\varphi)^{-1}(\cL_a^\epsilon - \lambda) = \partial_x^2 + D^\epsilon(\varphi)^{-1} b^\epsilon_{1,a}(x) \partial_x + D^\epsilon(\varphi)^{-1} (b^\epsilon_{0,a}(x) - \lambda ) \Idd ,
\end{equation}  
are the same (see Lemma \ref{lemJFred}).

Following a procedure which is customary in the literature on nonlinear waves (cf. \cite{AGJ90,San02,KaPro13}), let us recast the spectral problem \eqref{opJl} as a first order system of the form
\begin{equation}
\label{firstorder}
\bw_x = \A^\epsilon_a(x,\lambda) \bw,
\end{equation}
where
\[
\bw = \begin{pmatrix} u \\ u_x
\end{pmatrix} \in H^1 \times H^1,
\]
and,
\begin{equation}
\label{gencoeff}
\A^\epsilon_a (x,\lambda) = \begin{pmatrix} 0 & 1 \\ D^\epsilon(\varphi)^{-1} (\lambda - b^\epsilon_{0,a}(x)) & -D^\epsilon(\varphi)^{-1} b^\epsilon_{1,a}(x)
\end{pmatrix}.
\end{equation}

It is a well-known fact (see, e.g., Theorem 3.2 in \cite{MRS14}, as well as the references \cite{KaPro13} and \cite{San02}) that the associated first order operators
\[
\mathcal{T}_a^\epsilon(\lambda) =  \partial_x -\mathbb{A}^\epsilon_a(\cdot , \lambda), \qquad \mathcal{T}^\epsilon_a(\lambda) : L^2 \times L^2 \to L^2 \times L^2,
\]
with domain $\cD(\mathcal{T}_a^\epsilon(\lambda)) = H^1 \times H^1$ (and hence, densely defined in $L^2 \times L^2$),
are endowed with the same Fredholm properties as $\mathcal{J}^\epsilon_a(\lambda)$ and, consequently, as $\cL^\epsilon_a - \lambda$. Moreover, these Fredholm properties depend upon the hyperbolicity of the asymptotic coefficient matrices (cf. \cite{San02}) 
\[
\A^{\epsilon,\pm}_a(\lambda) = \lim_{x \to \pm\infty} \A^\epsilon_a (x,\lambda). 
\]
Denoting
\[
\begin{aligned}
b^{\epsilon,\pm}_{0,a} := \lim_{x \to \pm \infty} b^\epsilon_{0,a}(x) &= a^2 D^\epsilon(u_\pm) - ac + f'(u_\pm),\\
b^{\epsilon,\pm}_{1,a} := \lim_{x \to \pm \infty} b^\epsilon_{1,a}(x) &= c - 2aD^\epsilon(u_\pm),
\end{aligned}
\]
then the constant coefficients of the asymptotic systems are given by
\begin{equation}
\label{genasymcoeff}
\A^{\epsilon,\pm}_a(\lambda) = \begin{pmatrix}
0 & 1 \\ D^\epsilon(u_\pm)^{-1} (\lambda - b^{\epsilon,\pm}_{0,a}) & - D^\epsilon(u_\pm)^{-1} b^{\epsilon,\pm}_{1,a}
\end{pmatrix}.
\end{equation}

Let us denote $\pi^{\epsilon,\pm}_a (\lambda,z) = \det (\A^{\epsilon,\pm}_a(\lambda) - z \Id)$, so that, for each $k \in \R$
\[
\pi^{\epsilon,\pm}_a (\lambda,ik) = -k^2 + ik(c D^\epsilon(u_\pm)^{-1} - 2a) + a^2 - D^\epsilon(u_\pm)^{-1} (\lambda + ac - f'(u_\pm)).
\]
Thus, the Fredholm borders, defined as the $\lambda$-roots of $\pi^{\epsilon,\pm}_a(\lambda, ik) = 0$, are given by
\begin{equation}
\lambda^{\epsilon,\pm}_a(k) := D^\epsilon(u_\pm) (a^2 - k^2) -ac + f'(u_\pm) + ik(c - 2a D^\epsilon(u_\pm)), \quad k \in \R.
\end{equation}
Notice that
\[
\max_{k \in \R} \, \Re \lambda^{\epsilon,\pm}_a(k) = D^\epsilon(u_\pm) a^2 -ac + f'(u_\pm).
\]
Therefore, we define the region of consistent splitting for each $a \in \R$ and each $\epsilon \geq 0$ as
\begin{equation}
\label{Omegae}
\Omega(a,\epsilon) := \left\{ \lambda \in \C \, : \, \Re \lambda > \max \, \{ D^\epsilon(u_+) a^2 -ac + 
f'(u_+), \, D^\epsilon(u_-) a^2 -ac + f'(u_-)\} \right\}.
\end{equation}

Now, for each fixed $\lambda \in \C$, $a \in \R$ and $\epsilon > 
0$, let us denote $S_a^{\epsilon,\pm}(\lambda)$ and $U_a^{\epsilon,\pm}(\lambda)$ as the stable and unstable 
eigenspaces of $\A_a^{\epsilon,\pm}(\lambda)$ in $\C^2$, respectively.

\begin{lemma}[consistent splitting]
\label{lemconsplit}
For each $\lambda \in \Omega(a,\epsilon)$, $a \in \R$ and all $\epsilon > 0$, the coefficient matrices $\A_a^{\epsilon, \pm}(\lambda)$ 
have no center eigenspace and 
\[
\dim S_a^{\epsilon,\pm}(\lambda) = \dim U_a^{\epsilon,\pm}(\lambda) = 1.
\]
\end{lemma}
\begin{proof}
Follows by elementary facts: thanks to connectedness of $\Omega(a,\epsilon)$ and continuity of $\lambda$, it is clear that the dimensions of the stable and unstable eigenspaces do not change for $\lambda \in \Omega(a,\epsilon)$. Moreover, since the Fredholm curves are defined precisely as the set of points in the complex plane in which there is a center eigenspace, we conclude that the asymptotic matrices are strictly hyperbolic in $\Omega(a,\epsilon)$. To compute the dimensions of the eigenspaces it suffices to take $\lambda \in \R, \, \lambda \gg 1$, sufficiently large. Thus, we find that the characteristic equation $\pi_a^{\epsilon, \pm}(\lambda,z) = 0$ has one positive and one negative root:
\[
\begin{aligned}
z_1 &=  -\tfrac{1}{2} (D^\epsilon(u_\pm))^{-1} b^{\epsilon,\pm}_{1,a} - \tfrac{1}{2} \sqrt{ ((D^\epsilon(u_\pm))^{-1} b^{\epsilon,\pm}_{1,a})^2 + 4(D^\epsilon(u_\pm))^{-1}(\lambda - b^{\epsilon,\pm}_{0,a})} \; < 0,\\
z_2 &=  -\tfrac{1}{2} (D^\epsilon(u_\pm))^{-1} b^{\epsilon,\pm}_{1,a} + \tfrac{1}{2} \sqrt{ ((D^\epsilon(u_\pm))^{-1} b^{\epsilon,\pm}_{1,a})^2 + 4(D^\epsilon(u_\pm))^{-1}(\lambda - b^{\epsilon,\pm}_{0,a})} \;> 0,
\end{aligned}
\]
for $\lambda > 0$ large. The conclusion follows.
\end{proof}

\subsection{Fredholm properties and generalized convergence}

Next, we characterize the Fredholm properties of the operators $\cL_a^\epsilon - \lambda$ for $\lambda$ in the region of consistent splitting.

\begin{lemma}
\label{lemLep}
For all $\epsilon > 0$, $a \in \R$, and for each $\lambda \in \Omega(a,\epsilon)$, the operator $\cL_a^\epsilon - \lambda$ is Fredholm with index zero.
\end{lemma}
\begin{proof}
If we fix $\lambda \in \Omega(a,\epsilon)$ then the matrices $\A_a^{\epsilon,\pm}(\lambda)$ are hyperbolic. Therefore, by standard exponential dichotomies theory \cite{Cop78} (see also Theorem 3.3 in \cite{San02}), system \eqref{firstorder} is endowed with exponential dichotomies on both rays $[0,+\infty)$ and $(-\infty,0]$, with Morse indices $i_a^{\epsilon,+}(\lambda) = \dim U_a^{\epsilon,+}(\lambda) = 1$ and $i_a^{\epsilon,-}(\lambda) = \dim S_a^{\epsilon, -}(\lambda) = 1$, respectively.  Apply Theorem 3.2 in \cite{San02} to conclude that the operators $\mathcal{T}_a^{\epsilon}(\lambda)$ are Fredholm with index $\ind \mathcal{T}_a^\epsilon(\lambda) = i_a^{\epsilon,+}(\lambda) - i_a^{\epsilon,-}(\lambda) = 0$, showing that $\mathcal{J}_a^\epsilon(\lambda)$ and $\cL_a^\epsilon - \lambda$ are Fredholm with index zero, as claimed.
\end{proof}

For later use, it is convenient to define the set
\begin{equation}
\label{defOmegaa}
\begin{aligned}
\Omega(a) &:= \Omega(a,0) \\&= \left\{ \lambda \in \C \, : \, \Re \lambda > \max \, \{ D(u_+) a^2 -ac + 
f'(u_+), \, D(u_-) a^2 -ac + f'(u_-)\} \right\},
\end{aligned}
\end{equation} 
which is independent of $\epsilon > 0$ and is actually the region of consistent splitting (see section \S \ref{secexpwNd} below).

We now prove that in the generalized limit when $\epsilon \to 0^+$, the regularized operators conserve their Fredholm properties.  First, let us remind the reader the following standard definitions from Functional Analysis (cf. Kato \cite{Kat80}): Let $Z$ be a Banach 
space, and let $M$ and $N$ be any two nontrivial closed subspaces of $Z$. Let $S_M$ be the unitary sphere in 
$M$. Then one defines
\[
\begin{aligned}
\mathrm{d}(M,N) &:= \sup_{u\in S_M} \mathrm{dist}(u,S_N),\\
\mathrm{\hat{d}}(M,N) &:= \max \{\mathrm{d}(M,N),\mathrm{d}(N,M)\}.
\end{aligned}
\]
$\mathrm{\hat{d}}(M,N)$ is called the \textit{distance} between $M$ and $N$ and satisfies the triangle 
inequality. (Here the function $\mathrm{dist}(u, M)$ is the 
usual distance function from $u$ to any closed manifold $M$.)

\begin{definition}
Let $X,Y$ be Banach spaces. If $\mathcal{T}, \mathcal{S} \in \mathscr{C}(X,Y)$, then the graphs 
$G(\mathcal{T}), G(\mathcal{S})$ 
are closed subspaces of $X\times Y$, and we set 
\begin{equation*}
\mathrm{d}(\mathcal{T},\mathcal{S}) = \mathrm{d}(G(\mathcal{T}), G(\mathcal{S})) ,
\end{equation*}
\begin{equation*}
\mathrm{\hat{d}}(\mathcal{T},\mathcal{S}) = \max 
\{\mathrm{d}(\mathcal{T},\mathcal{S}),\mathrm{d}(\mathcal{S},\mathcal{T})\}.
\end{equation*}
It is said that a sequence $\mathcal{T}_n \in \mathscr{C}(X,Y)$ \textit{converges in generalized sense} to 
$\mathcal{T}  \in \mathscr{C}(X,Y)$ provided that 
$\mathrm{\hat{d}}(\mathcal{T}_n,\mathcal{T}) \to 0$ as $n \to +\infty$.
\end{definition}

\begin{lemma}
\label{lemconv}
For each fixed $\lambda \in \C$ and each $a \in \R$, the operators $\cL_a^\epsilon - \lambda$ converge in generalized sense to $\cL_a - \lambda$ as $\epsilon \to 0^+$.
\end{lemma}
\begin{proof}
It is to be observed that, from the definition of $\mathrm{d}(\cdot,\cdot)$, we have
\begin{equation*}
\mathrm{d}(\cL_a^\epsilon -\lambda,\cL_a -\lambda) = \mathrm{d}(G(\cL_a^\epsilon 
-\lambda),G(\cL_a -\lambda)) = \sup_{v\in S_{G(\cL_a^\epsilon -\lambda)}} \!\!
\left ( \inf_{w\in S_{G(\cL_a -\lambda)}} \| v-w \|  \right ),
\end{equation*}
where $S_{G(\cL_a -\lambda)}$ (respectively, $S_{G(\cL^\epsilon_a -\lambda)}$) is the unit sphere in the graph $G(S_{G(\cL_a -\lambda)})$ (respectively, $G(S_{G(\cL^\epsilon_a -\lambda)})$). Now let $v\in S_{G(\cL_a^\epsilon -\lambda)}$ be such that $v = (p,(\cL_a^\epsilon -\lambda)p) \in L^2 \times L^2$ for 
some $p\in \mathcal{D}(\cL_a^\epsilon -\lambda) = H^2$ and with
\begin{equation*}
\| v\| ^2_{L^2 \times L^2} = \| p \| ^2_{L^2} + \| (\cL_a^\epsilon -\lambda)p \| ^2_{L^2} =1.
\end{equation*}
Likewise, let $w\in S_{G(\cL_a -\lambda)}$ be such that $w = (u,(\cL_a -\lambda)u)$ for some $u\in 
\mathcal{D}(\cL_a -\lambda) = H^2$ and
\begin{equation*}
\| w\| ^2_{L^2 \times L^2} = \| u \| ^2_{L^2} + \| (\cL_a -\lambda)u \| ^2_{L^2} =1.
\end{equation*}

Let us find a upper bound for $\| v-w\|_{L^2 \times L^2}$ so that the set consisting of the infima taken over $S_{G(\cL_a -\lambda)}$ is bounded and therefore the supremum over $S_{G(\cL^\epsilon_a -\lambda)}$ is bounded too. Let us write 
\begin{equation}
\label{graphNorm}
\| v-w \| ^2_{L^2 \times L^2} = \| p-u \| ^2_{L^2} +\| (\cL_a^\epsilon -\lambda)p -(\cL_a 
-\lambda)u \| ^2_{L^2}.
\end{equation}
If we keep $v \in S_{G(\cL_a^\epsilon -\lambda)}$ fixed then it suffices to take $w = (p,(\cL_a 
-\lambda)p)$ because $p \in H^2 = \mathcal{D}(\cL_a) = \mathcal{D}(\cL_a^\epsilon)$. Note 
that $(\cL_a -\lambda)p = (\cL_a^\epsilon -\lambda)p - \epsilon p_{xx}$. Therefore expression 
\eqref{graphNorm} gets simplified to
\begin{equation*}
\| v-w \| ^2_{L^2 \times L^2} = \| (\cL_a^\epsilon -\lambda)p -(\cL_a -\lambda)p \| ^2_{L^2} = \| 
\epsilon p_{xx} \| ^2_{L^2}.
\end{equation*}

If we regard $\partial_x^2$ as a closed, densely defined operator on $L^2$, with domain $\cD = 
H^2$, then it follows from Remark 1.5 in \cite[p. 191]{Kat80}, that $\partial_x^2$ is 
$(\cL_a^\epsilon -\lambda)-$bounded, i.e., there exist a constant $C_\lambda>0$ such that 
\begin{equation*}
\| p_{xx} \|_{L^2} \leq C_\lambda (\| p \| _{L^2} + \| (\cL_a^\epsilon -\lambda)p \| _{L^2} ),
\end{equation*}
for all $p \in H^2$. Observe that the constant $C_\lambda$ may depend on $\lambda \in \C$, fixed. Consequently,
\[ \| p_{xx} \|^2_{L^2} \leq \bar{C}_\lambda (\| p \| ^2_{L^2}+\| (\cL_a^\epsilon -\lambda)p \| 
^2_{L^2})=\bar{C}_\lambda,
\]
for some other $\bar{C}_\lambda > 0$ and for $v = (p, (\cL^\epsilon - \lambda)p) \in S_{G(\cL_a^\epsilon 
-\lambda)}$. This estimate implies, in turn, that
\[ 
\| v-w \| ^2_{L^2 \times L^2} = \epsilon^2 \| p_{xx} \| ^2_{L^2} \leq \bar{C}_\lambda \epsilon^2,
\]
yielding
\begin{equation*}
\mathrm{d}(\cL_a^\epsilon -\lambda,\cL_a -\lambda) \leq \bar{C}_\lambda\epsilon^2.
\end{equation*}
In a similar fashion it can be proved that 
$\mathrm{d}(\cL_a -\lambda,\cL_a^\epsilon -\lambda) = O(\epsilon^2)$.
This shows that $\hat{\mathrm{d}}(\cL_a^\epsilon -\lambda,\cL_a -\lambda) \to 0$
as $\epsilon\to 0 ^{+}$ and the lemma is proved.
\end{proof}

Now, let us recall the definition of the reduced minimum modulus of a closed operator: for any $\mathcal{S}\in \mathscr{C}(X,Y)$ we define $\gamma= \gamma(\mathcal{S})$ as the greatest number 
$\gamma \in \mathbb{R}$ such that  
\[ 
\| \mathcal{S}u \|  \geq \gamma \: \mathrm{dist}(u, \ker \mathcal{S}), 
\]
for all $u \in  \mathcal{D}(\mathcal{\mathcal{S}})$. It is well-known that $\cR(\cS)$ is closed if and only if $\gamma(\cS) > 0$ (cf. \cite{EE87,Kat80}). The following classical result of functional analysis due to Kato, known as the general stability theorem, is the main tool to relate the Fredholm properties of $\cL_a^\epsilon - \lambda$ to those of $\cL _a- \lambda$ (see \cite{Kat80}, p. 235).

\begin{theorem}[Kato's stability theorem]
\label{thmindexKato}
Let $\mathcal{T}, \mathcal{S} \in \mathscr{C}(X,Y)$ and let $\mathcal{T}$ be Fredholm (semi-Fredholm). If 
\[
\mathrm{\hat{d}}(\mathcal{S},\mathcal{T})< \gamma (1+ \gamma^2)^{-1/2},
\] 
where $\gamma=\gamma(\mathcal{T})$, then $\mathcal{S}$ is Fredholm (semi-Fredholm) and $\nul \mathcal{S} 
\leq \nul \mathcal{T}$, $\mathrm{def}\, \mathcal{S} \leq \mathrm{def}\, \mathcal{T}$. Furthermore, there 
exists 
$\delta> 0$ such that $\mathrm{\hat{d}}(\mathcal{S},\mathcal{T})<\delta$ implies 
\[ 
\ind \mathcal{S} = \ind \mathcal{T}. 
\]
If $X,Y$ are Hilbert spaces then we can take $\delta = \gamma (1+ \gamma^2)^{-\frac{1}{2}}$.
\end{theorem}
\begin{remark}
The original statement of Kato's theorem is expressed in terms of the \emph{gap} between the graphs; the distance and the gap are, however, equivalent (see Kato \cite{Kat80}, p. 202).
\end{remark}

We apply Kato's theorem in order to obtain the following result, which is the main tool to locate the subset of the compression spectrum, $\sigma_\delta(\cL_a)$, in the degenerate front case.
\begin{lemma}
\label{lemLsemi}
Suppose that $\cL_a - \lambda$ is semi-Fredholm, for $a \in \R$, $\lambda \in \C$. Then for each $0 < \epsilon \ll 1$ sufficiently small $\cL^\epsilon_a - \lambda$ is semi-Fredholm and $\ind(\cL^\epsilon_a - \lambda) = \ind(\cL_a - \lambda)$.
\end{lemma}
\begin{proof}
If $\cL_a - \lambda$ is semi-Fredholm then $\cR(\cL_a - \lambda)$ is closed and, consequently, $\gamma : = \gamma(\cL_a-\lambda) > 0$ and is independent of $\epsilon > 0$. Since $\cL_a^\epsilon - \lambda \, \to \, \cL_a - \lambda$ in the generalized sense as $\epsilon \to 0^+$, then we have that for all $0 < \epsilon \ll 1$ sufficiently small, 
\[
\mathrm{\hat{d}}(\cL_a^\epsilon - \lambda, \cL_a - \lambda) <  \gamma (1 + \gamma^2)^{-1/2}.
\]
From a direct application of Kato's theorem \ref{thmindexKato} we conclude that $\cL^\epsilon_a - \lambda$ is semi-Fredholm and $\ind(\cL^\epsilon_a - \lambda) = \ind(\cL_a - \lambda)$, for all $\epsilon > 0$ small enough.
\end{proof}

\section{Energy estimates on spectral equations}
\label{secenergyest}

In this section we establish the basic energy estimate for solutions to the spectral equation. The monotonicity of the fronts plays a crucial role in the analysis.

Let $\varphi= \varphi(x)$ be any of the monotone fronts of Proposition \ref{propstructure} traveling with speed $c
 \geq 0$, and let $\cL_a : \cD(\cL_a) = H^{2} \subset L^2 \to L^2$ be the conjugated linearized operator around the front \eqref{conjopL} for a certain fixed $a \in \R$. Take any fixed $\lambda \in \C$ and suppose that there exists a solution $u\in  \cD(\cL_a) = H^{2}$ to the spectral equation
\begin{equation}
\label{spectralp2}
\begin{aligned}
0 = (\cL_a - \lambda)u &= D(\varphi)u_{xx} + \Big( 2 D(\varphi)_x - 2a D(\varphi) + c\Big) u_x + \\ & \quad + \Big( a^2 D(\varphi) - 2a D(\varphi)_x - ac + D(\varphi)_{xx} + f'(\varphi) - \lambda  \Big)u.
\end{aligned}
\end{equation}
Consider the change of variables
\begin{equation}
\label{changev}
u(x)= w(x) e^{\theta(x)},
\end{equation}
where $\theta =\theta(x)$ is a function to be chosen later. Upon substitution into the spectral equation \eqref{spectralp2} we obtain the following equation for $w$,
\begin{equation}\label{eqwp}
D(\varphi)w_{xx} + \Big( 2 D(\varphi)_x + 2 \theta_x D(\varphi) - 2aD(\varphi) + c\Big) w_x+ \widetilde{G}(x,a)w - \lambda w = 0,
\end{equation}
where the coefficient
\[
\begin{aligned}
\widetilde{G}(x,a) &:= D(\varphi)\big( \theta_{xx} + \theta_x^2 \big) + \big( 2 D(\varphi)_x -a D(\varphi) + c \big) \theta_x \\ &\quad  + a^2 D(\varphi) - 2 aD(\varphi)_x - ac + D(\varphi)_{xx} + f'(\varphi),
\end{aligned}
\]
is a function of $x \in \R$, parametrized by $a \in \R$ and $c \geq 0$. Let us choose $\theta = \theta(x)$ such that
\[
2 \theta_x D(\varphi) - 2aD(\varphi) + c = 0,
\]
for all $x \in \R$. In view that $D(\varphi) > 0$ for all $x \in \R$ and that it may vanish only in the limit when $x \to \pm \infty$, we obtain
\begin{equation}\label{theta}
\theta(x) = -\frac{c}{2} \int_{x_0}^{x}\frac{dy}{D(\varphi(y))} + a(x-x_0), \qquad x \in \R,
\end{equation}
for any fixed (but arbitrary) $x_0 \in \R$. With this choice of $\theta = \theta(x)$, and upon substitution into \eqref{eqwp}, one gets
\begin{equation}
\label{eqw}
D(\varphi)w_{xx}+2 D(\varphi)_x w_x+ G(x)w - \lambda w = 0,
\end{equation}
where 
\[
G(x) := \widetilde{G}(x,a) = -\frac{c}{2} \frac{D(\varphi)_x}{D(\varphi)}-\frac{c^2}{4 D(\varphi)}+ D(\varphi)_{xx}+f'(\varphi).
\]
In other words, with this particular form of $\theta$, the function $\widetilde{G}(x,a)$ becomes independent of the parameter $a \in \R$. Notice that $\theta = \theta(x)$ is well defined for all $x \in \R$; depending on the case under consideration, though, it may diverge (and consequently, so may $e^{-\theta(x)}$) as $x \to \pm \infty$.

In the analysis that follows, we are going to suppose that $u \in H^2$ decays sufficiently fast as $x \to \pm \infty$ so that $w$ defined in \eqref{changev} belongs to $H^2$ as well. More precisely, we make the following
\begin{assumption}
\label{assumcrucial}
Given $a \in \R$, for any $\lambda \in \Omega(a)$ and any solution $u \in H^2$ of the spectral equation $(\cL_a - \lambda) u = 0$, the function
\begin{equation}
\label{defofchangew}
w(x) = \exp \left(\frac{c}{2} \int_{x_0}^{x}\frac{dy}{D(\varphi(y))} - a(x-x_0) \right) u(x), \quad x \in \R,
\end{equation}
belongs to $H^2$ as well. The set $\Omega(a)$ is defined in \eqref{defOmegaa}.
\end{assumption}

\begin{lemma}[basic energy estimate]\label{lembee}
For $a \in \R$ fixed let $\cL_a : \cD = H^2 \subset L^2 \to L^2$ be the conjugated linearized operator around any monotone traveling Nagumo front of Proposition \ref{propstructure}. Let $\lambda \in \Omega(a)$ be fixed and let $u \in H^2$ be any solution to the spectral equation $(\cL_a - \lambda) u = 0$. Let us suppose that $a \in \R$ is such that $0 \in \Omega(a)$ and $e^{ax} \varphi_x \in H^2$. Then, under Assumption \ref{assumcrucial} there holds the energy estimate
\begin{equation}
 \label{basicee}
\lambda \< D(\varphi)w, w\>_{L^2} = -  \| D(\varphi) \psi (w/\psi)_x \|_{L^2}^2,
\end{equation}
where $w := e^{-\theta(x)}u \in H^2$, $\psi := e^{-\theta(x)} e^{ax}\varphi_x \in H^2$ and $\theta = \theta(x)$ is defined in \eqref{theta}.
\end{lemma}
\begin{proof}
By hypothesis, $a \in \R$ is such that $\lambda = 0 \in \Omega(a)$ and $e^{ax} \varphi_x \in H^2$. Hence, since $\varphi_x \in H^2$ is an $L^2$-eigenfunction of the linearized operator $\cL$ around the wave associated to the zero eigenvalue, as $\cL \varphi_x = 0$ (see Remark \ref{remef0}), we obtain that $e^{ax} \varphi_x \in H^2$ is an $L^2$-eigenfunction of the conjugated operator associated to the zero eigenvalue as well, because
\[
\cL_a e^{ax} \varphi_x = e^{ax} \cL e^{-ax} e^{ax} \varphi_x = e^{ax} \cL \varphi_x = 0.
\]

Thanks to Assumption \ref{assumcrucial}, the solution $u = e^{ax} \varphi_x$ to the spectral equation $\cL_a u = 0$ is such that $\psi = e^{-\theta(x)} e^{ax} \varphi_x \in H^2$, in view that $\lambda = 0 \in \Omega(a)$. Hence, we can repeat the procedure leading to equation \eqref{eqw} on $\psi$ to arrive at the equation
\begin{equation}
\label{eqpsi}
D(\varphi)\psi_{xx}+ 2D(\varphi)_x \psi_x+G(x)\psi =0.
\end{equation}

Now take any $\lambda \in \Omega(a)$ and any solution $u \in H^2$ to the spectral equation $(\cL_a - \lambda) u = 0$. Defining $w = e^{-\theta(x)}u \in H^2$ (by Assumption \ref{assumcrucial}) we arrive at the spectral equation \eqref{eqw}. Multiply equations \eqref{eqw} and \eqref{eqpsi} by $D(\varphi)$ and rearrange the terms appropriately. The result is
\begin{equation}
 \label{eqab}
 \begin{aligned}
  (D(\varphi)^{2}w_x)_x +D(\varphi)G(x)w - \lambda D(\varphi)w &= 0,\\
  (D(\varphi)^{2}\psi_x)_x +D(\varphi)G(x)\psi &= 0.
 \end{aligned}
\end{equation}
Since the fronts is monotone ($\varphi_x \gtrless 0$) we have that $\psi \neq 0$ for all $x \in \R$ and we can substitute
\[ 
D(\varphi)G(x) = -\frac{(D(\varphi)^{2}\psi_x)_x}{\psi}
\]
into the first equation in \eqref{eqab} to obtain
\begin{equation}
\label{sltype}
(D(\varphi)^{2}w_x)_x -\frac{(D(\varphi)^{2}\psi_x)_x}{\psi}w - \lambda D(\varphi)w = 0.
\end{equation}

Take the complex $L^2$-product of $w$ with last equation and integrate by parts. This yields,
\[
\begin{aligned}
\lambda \int\limits_{\R} D(\varphi)|w|^{2} dx  &= \int\limits_{\R} 
\overline{w} 
(D(\varphi)^{2}w_x)_x dx - \int\limits_{\R} (D(\varphi)^{2}\psi_x)_x \frac{|w|^{2}}{\psi} dx\\
&= -\int\limits_{\R} D(\varphi)^{2}|w_x|^{2}dx + \int\limits_{\R} D(\varphi)^{2} \psi_x 
\left(\frac{|w|^{2}}{\psi} \right)_x dx \\
&= \int\limits_{\R} D(\varphi)^{2} \left( \psi_x \left(\frac{|w|^{2}}{\psi} \right)_x-|w_x|^{2} \right) dx.
\end{aligned}
\]
Use the identity 
\[
\psi^{2}\left|\left(\frac{w}{\psi} \right)_x \right|^{2} = - \left( \psi_x \left(\frac{|w|^{2}}{\psi} \right)_x-|w_x|^{2} \right),
\]
to obtain
\[
\lambda \int\limits_{\R} D(\varphi)|w|^{2} dx  = -  \int\limits_{\R}  D(\varphi)^{2}\psi^{2}\left|\left(\frac{w}{\psi} \right)_x \right|^{2} dx,
\]
as claimed. Finally, we need to verify that $D(\varphi) \psi (w/\psi)_x \in L^2$. But this is a straightforward consequence of Assumption \ref{assumcrucial}. Indeed, use the expression for $\theta$ and substitute $\psi = e^{-\theta(x)} e^{ax} \varphi_x$ to obtain 
\[
D(\varphi) \psi \left(\frac{w}{\psi} \right)_x = \frac{c}{2} w + D(\varphi) \frac{\varphi_{xx}}{\varphi_x} w,
\]
yielding
\[
\| D(\varphi) \psi (w/\psi)_x \|_{L^2}\leq C \| w \|_{L^2} < \infty,
\]
for some uniform $C > 0$ because the coefficient $D(\varphi)\varphi_{xx}/\varphi_x$ is uniformly bounded for all $x \in \R$ (see Lemma \ref{lemauxi}) and $w \in H^2$ by Assumption \ref{assumcrucial}. Thus the integral on the right hand side of \eqref{basicee} exists and the conclusion follows.
\end{proof}

\begin{remark}
A few observations are in order. First, notice that we need to find $a \in \R$ such that Assumption \ref{assumcrucial} holds in order to obtain the energy estimate. Hence, the search for such $a \in \R$ will be done on a case by case basis. Second, observe that the monotonicity of the front is important to conclude. Not only is it crucial to eliminate one of the spectral equations and to perform the energy estimate, but it is a fundamental property of the trajectory: it is well-known, in the particular case of constant diffusion coefficients, for instance, that the traveling pulse solutions to bistable reaction-diffusion equations are spectrally unstable precisely because of the lack of monotonicity of the trajectory, which renders positive eigenvalues of second order differential operators of Sturm type (see Section 2.3.3 and Theorem 2.3.3 in \cite{KaPro13}).
\end{remark}
\begin{remark}
The transformation \eqref{changev} of the eigenfunction and the corresponding energy estimate that leads to the conclusion of Lemma \ref{lembee} has been also adapted to the spectral stability of planar fronts with degenerate diffusions in a multidimensional setting \cite{BMP17} and to traveling fronts for hyperbolic models of diffusion \cite{LMPS-book}.
\end{remark}

\section{Spectral stability of monotone degenerate fronts}
\label{secNd}

In this section, we establish the spectral stability property of the family of monotone, non-stationary, diffusion-degenerate Nagumo fronts, for which
\[
\begin{aligned}
u_+ &= \alpha, &  & u_- = 0,\\
f'(u_+) = f'(\alpha) &> 0, & & f'(u_-) = f'(0) < 0,
\end{aligned}
\]
which correspond to fronts of Type I in Proposition \ref{propstructure}. Each front $\varphi = \varphi(x)$ is monotone increasing, $\varphi_x > 0$, traveling with positive constant speed
\[
c > \bbc(\alpha) = 2 \sqrt{D(\alpha) f'(\alpha)} > 0.
\]

The fronts are diffusion degenerate in view that $D(u_-) = D(0) = 0$ and the underlying spectral problem is degenerate. Consequently, we use the partition of spectrum of Definition \ref{defspecd}. In addition, note that $f'(u_+) = f'(\alpha) > 0$, $f'(u_-) = f'(0) < 0$ and, thus, the unweighted continuous $L^2$-spectrum of the linearized operator $\cL$ around the front is unstable. Therefore, it is necessary to locate the subset of the compression spectrum, $\sigma_\delta$, under both conjugation and parabolic regularization. Section \ref{secexpwNd} is devoted to find  the sufficient conditions on the exponential weight so that $\sigma_\delta(\cL)_{|L^2_a}$ gets stabilized. These conditions on $a \in \R$ are crucial in order to verify Assumption \ref{assumcrucial}. For that purpose, Section \ref{secdecaystr} establishes the detailed decay structure of solutions to the spectral equation on both the degenerate and the non-degenerate sides. On the non-degenerate side, as $x \to \infty$, the asymptotic coefficient is hyperbolic and we are able to apply \emph{the Gap Lemma} (cf. \cite{GZ98,KS98}) to obtain the exact decay rate of the eigenfunctions. Due to the degeneracy in the asymptotic limit as $x \to -\infty$ and after a detailed description of the eigenfunctions' decay structure, we find the necessity to impose a further condition on the exponential weight in order to verify Assumption \ref{assumcrucial}. In this fashion, under the appropriate exponential weight compatible with the hypotheses of Lemma \ref{lembee}, we prove that the point spectrum of the conjugated operator is stable (Section \ref{secptspNd}). Finally, thanks also to the choice of the proper energy space $L^2_a$, in Section \ref{seclocapp} we prove the stability of the subset of spectrum $\sigma_\pi$ by showing that singular sequences disperse their $L^2$-mass asymptotically to $\pm \infty$ where the sign of the coefficients of the operator is known. An energy estimate closes the argument.

\subsection{Choice of the exponential weight}
\label{secexpwNd}

First, we choose the weight $a \in \R$ that defines the conjugated operator $\cL_a$ and stabilizes $\sigma_\delta(\cL)_{|L^2_a} = \sigma_\delta(\cL_a)_{|L^2}$. In this case the region of consistent splitting (see \eqref{defOmegaa} above) is given by
\[
\Omega(a) = \big\{ \lambda \in \C \, : \, \Re \lambda > \max \{ D(\alpha)a^2 - ac + f'(\alpha), \, - ac + f'(0) \} \, \big\},
\]
for each $a \in \R$. 

\begin{lemma}
\label{lemdspecloc}
For any $a > 0$, there holds
\[
\sigma_\delta(\cL_a)_{|L^2} \subset \C \backslash \Omega(a).
\]
\end{lemma}
\begin{proof}
Consider the regularized operator $\cL_a^\epsilon : \cD(\cL_a^\epsilon) = H^2 \subset L^2 \to L^2$ defined in \eqref{conjop} (Section \ref{secregop}) for some $\epsilon > 0$ and $a > 0$. Let us suppose, by contradiction, that there exists $\lambda_0 \in \sigma_\delta(a) \cap \Omega(a)$. First, observe that $\cL_a - \lambda_0$ is semi-Fredholm with $\ind(\cL_a - \lambda_0) \neq 0$. Indeed, if $\lambda_0 \in \sigma_\delta(\cL_a)$ then, by definition, $\cL_a - \lambda_0$ is injective, $\cR(\cL_a - \lambda_0)$ is closed and $\cR(\cL_a - \lambda_0) \subsetneqq L^2$ (see Definition \ref{defspecd}). This implies that $\nul (\cL_a - \lambda_0) = 0$ and $\cL_a - \lambda_0$ is semi-Fredholm. Moreover, since $\mathrm{def} (\cL_a - \lambda_0) = \mathrm{codim} \cR(\cL_a - \lambda_0) > 0$ we conclude that $\ind(\cL_a - \lambda_0) \neq 0$.

Now, thanks to Lemma \ref{lemconv}, we deduce the existence of $\epsilon_0 > 0$ sufficiently small such that $\cL_a^\epsilon - \lambda_0$ is semi-Fredholm with $\ind(\cL_a^\epsilon - \lambda_0) = \ind(\cL_a - \lambda_0) \neq 0$ for each $0 < \epsilon < \epsilon_0$. Now, since $\lambda_0 \in \Omega(a)$ we have
\[
\Re \lambda > q_0(\alpha,a) := \max \{ D(\alpha)a^2 - ac + f'(\alpha), \, - ac + f'(0) \}.
\]

In view that $D^\epsilon(u_\pm) -ac + f'(u_\pm) = \epsilon a^2 + D(u_\pm) -ac + f'(u_\pm)$ and $a > 0$, we can always choose $\widetilde{\epsilon} > 0$ sufficiently small such that
\[
0 < \widetilde{\epsilon} < \tfrac{1}{2} \min \left\{ \epsilon_0, \frac{\Re \lambda_0 - q_0(\alpha,a)}{a^2}\right\}.
\]
This implies that $\lambda_0 \in \Omega(a,\widetilde{\epsilon})$ by definition (see \eqref{Omegae}). Apply Lemma \ref{lemLep} to conclude that $\cL_a^{\widetilde{\epsilon}} - \lambda_0$ is Fredholm with 
$\ind(\cL_a^{\widetilde{\epsilon}} - \lambda_0) = 0$, yielding a contradiction. This proves the lemma.
\end{proof}

Henceforth, in order to achieve stability of $\sigma_\delta(\cL_a)_{|L^2}$ we need to find values of $a > 0$ such that $q_0(\alpha,a) < 0$. In view that $f'(0) < 0$, this happens if and only if $a_1(\alpha) < a < a_2(\alpha)$, where $a_1(\alpha)$ and $a_2(\alpha)$ are the roots of 
\begin{equation}
\label{defpaalpha}
p(a;\alpha) := D(\alpha) a^2 - ac + f'(\alpha) = 0,
\end{equation}
namely,
\[
\begin{aligned}
a_1(\alpha) &=  \frac{1}{2D(\alpha)} \Big( c - \sqrt{c^2 - 4D(\alpha) f'(\alpha)} \, \Big),\\
a_2(\alpha) &=  \frac{1}{2D(\alpha)} \Big( c + \sqrt{c^2 - 4D(\alpha) f'(\alpha)} \, \Big).
\end{aligned}
\]
Observe that $0 < a_1(\alpha) < a_2(\alpha)$, inasmuch as $c > \bbc(\alpha)$. We conclude that it suffices to take $0 < a_1(\alpha) < a < a_2(\alpha)$ to stabilize $\sigma_\delta(\cL_a)_{|L^2}$. We have thus proved the following
\begin{lemma}
\label{lemsdNd}
For each value of $a \in \R$ such that
\begin{equation}
\label{condaNdprel}
0 < a_1(\alpha) < a < a_2(\alpha),
\end{equation}
the set $\sigma_\delta(\cL_a)_{|L^2}$ is stable, more precisely,
\[
\sigma_\delta(\cL_a)_{|L^2} \subset \C \backslash \Omega(a) = \{ \lambda \in \C \, : \, \Re \lambda \leq - \mu_0 < 0\},
\]
where $\mu_0 = - \max \{ D(\alpha)a^2 - ac + f'(\alpha), \, - ac + f'(0) \} > 0$.
\end{lemma}

\subsection{Decay structure of solutions to spectral equations}
\label{secdecaystr}

In this section we show that there exists $a \in \R$ satisfying condition \eqref{condaNdprel} for which, in addition, Assumption \ref{assumcrucial} holds. This implies, in turn, that the change of variables \eqref{defofchangew} satisfies the hypotheses of Lemma \ref{lembee} in the present degenerate case. For that purpose, we examine the decay properties of solutions to the spectral equation.

Let us suppose that for a given fixed $a \in \R$ and $\lambda \in \Omega(a)$ there exists a solution $u \in H^2$ to the spectral equation  $(\cL_a - \lambda) u = 0$. Let us write the change of variables \eqref{defofchangew} as
\[
w(x) = e^{-\theta(x)}u(x) = \Theta(x) u(x), \qquad x \in \R,
\]
where
\begin{equation}
\label{defTheta}
\Theta(x) := \exp \left(\frac{c}{2} \int_{x_0}^{x}\frac{ds}{D(\varphi(s))} - a(x-x_0) \right) u(x), \quad x \in \R,
\end{equation}
and $\theta = \theta(x)$ is defined in \eqref{theta}. Here $x_0 \in \R$ is fixed but arbitrary. Since $\Theta$ is smooth, it is clear that if $u \in H^2$ then $w \in H^2(-x_0,x_0)$ for any finite $x_0 \in \R$. Hence, in order to show that $w \in H^2$, it suffices to prove that $w$ decays sufficiently fast as $x \to \pm \infty$. We shall show that if we impose further conditions on $a$ then we can guarantee this behavior.

Let us first examine the decay on the non-degenerate side as $x \to +\infty$.

\subsubsection{Behavior at $+\infty$}

The asymptotic decay properties of $u \in H^2$ as a solution to $(\cL_a - \lambda) u = 0$ are determined by those of $\bw \in H^1 \times H^1$, solution to
\begin{equation}
\label{trueqs}
\bw_x = \A_a(x,\lambda) \bw, \qquad \bw = \begin{pmatrix} u \\ u_x \end{pmatrix},
\end{equation}
where the variable coefficient matrices are given by
\[
\A_a(x,\lambda) = \begin{pmatrix} 0 & 1 \\ D(\varphi)^{-1} (\lambda - b_{0,a}(x)) & -D(\varphi)^{-1} b_{1,a}(x)
\end{pmatrix},
\]
with
\[
\begin{aligned}
b_{1,a}(x) &= 2D(\varphi)_x + c - 2a D(\varphi),\\
b_{0,a}(x) &= a^2D(\varphi) - 2aD(\varphi)_x -ac + D(\varphi)_{xx} + f'(\varphi)
\end{aligned}
\]
(see expressions \eqref{firstorder} - \eqref{gencoeff} with $\epsilon = 0$). On the non-degenerate side the asymptotic constant coefficients are well-defined and hyperbolic,
\begin{equation}
\label{asymANd}
\A_a^+(\lambda) = \lim_{x \to +\infty} \A_a(x,\lambda) = \begin{pmatrix} 0 & 1 \\ & \\ \displaystyle{\frac{\lambda - p(a;\alpha)}{D(\alpha)}} & \displaystyle{2a - \frac{c}{D(\alpha)}}
\end{pmatrix}, 
\end{equation}
and the associated asymptotic system is
\begin{equation}
\label{asymsystNd}
\bv_x = \A_a^+(\lambda) \bv.
\end{equation}

To determine the decaying properties of $\bw$ we invoke the Gap Lemma of Gardner and Zumbrun \cite{GZ98} and Kapitula and Sandstede \cite{KS98} (see also \cite{MZ03,MZ02,PZ04,Z6,ZH} for its different versions, further information and direct applications), which relates the decay structure of solutions of the variable coefficient system \eqref{trueqs} to those of the asymptotic constant coefficient system \eqref{asymsystNd}, provided that $\A_a(x,\lambda)$ approaches $\A_a^+(\lambda)$ exponentially fast as $x \to +\infty$. For the precise statement of the Gap Lemma we refer the reader to Lemma A.11 in \cite{Z6} or Appendix C in \cite{MZ02}.

First we need to prove the following auxiliary result.
\begin{lemma}
Given $a > 0$ suppose that $\lambda \in \Omega(a)$. Then the decaying eigenvalue of $\A_a^+(\lambda)$ is given by
\begin{equation}
\label{decmode}
\mu_1^+(\lambda) = a - \frac{c}{2D(\alpha)} - \frac{1}{2} \zeta^+(\lambda)^{1/2},
\end{equation}
where the discriminant
\begin{equation}
\label{discri}
\zeta^+(\lambda) := \Big( \frac{c}{D(\alpha)} - 2a\Big)^2 +  \frac{4(\lambda - p(a;\alpha))}{D(\alpha)},
\end{equation}
is an analytic function of $\lambda \in \Omega(a)$ and $p(a;\alpha)$ is defined in \eqref{defpaalpha}.
\end{lemma}
\begin{proof}
By a straightforward computation one easily finds that the eigenvalues of \eqref{asymANd} are
\[
\mu_{j}^+(\lambda) = a - \frac{c}{2D(\alpha)} + z_{j}^+(\lambda), \qquad j=1,2,
\]
where
\[
z_{1}^+(\lambda) = - \frac{1}{2}\zeta^+(\lambda)^{1/2}, \quad  z_{2}^+(\lambda) = \frac{1}{2}\zeta^+(\lambda)^{1/2}.
\]
Clearly, the discriminant $\zeta^+(\lambda)$ is linear in $\lambda$ and, hence, analytic in $\lambda \in \Omega(a)$. Therefore, here $\zeta^+(\lambda)^{1/2}$ denotes the principal branch of the square root and it is analytic in $\lambda \in \Omega(a)$ as well.

Now, we know from Lemma \ref{lemconsplit} that for $\lambda \in \R$, $\lambda \gg 1$ sufficiently large, the only eigenvalue with negative real part is $\mu_1^+(\lambda)$. Since the set of sufficiently large real numbers is contained in the connected set $\Omega(a)$ (this is evident by definition of $\Omega(a)$), and since the eigenvalues are continuous (actually, analytic) in $\lambda \in \Omega(a)$, we conclude that $\Re \mu_1^+(\lambda) < 0$ for all $\lambda \in \Omega(a)$ and that $\mu_1^+(\lambda)$ is the only decaying mode of $\A_a^+(\lambda)$. Otherwise the hyperbolicity (and consequently, the consistent splitting at $+\infty$) would be violated. This shows the result.
\end{proof}

It is to be observed that, thanks to the exponential decay of the front at the non-degenerate side (see Lemma \ref{lemdecayNd}), one readily obtains
\[
|\A_a(x,\lambda) - \A_a^+(\lambda)| \leq Ce^{-\eta x},
\]
as $x \to +\infty$ for some $C, \eta > 0$ and uniformly in $\lambda \in \Omega(a)$ (due to hyperbolicity). Henceforth, we are able to apply the Gap Lemma and to conclude that the decaying solution to the variable coefficient equation \eqref{trueqs} behaves as
\[
\bw(x,\lambda) = e^{\mu_1^+(\lambda)x} \big( \bv_1^+(\lambda) + O(e^{-\eta |x|} |\bv_1^+(\lambda)|) \big), \qquad x > 0,
\] 
where $ \bv_1^+(\lambda)$ denotes the eigenvector of $\A_a^+(\lambda)$ associated to the decaying eigenmode $\mu_1^+(\lambda)$ with $\Re \mu_1^+(\lambda) < 0$. This implies, in turn, that $u$ and $u_x$ decay at most as
\begin{equation}
\label{Nddecayuux}
|u|, |u_x| \leq C_1 e^{\Re \mu_1^+(\lambda) x} \to 0,
\end{equation}
for some $C_1 > 0$ as $x \to +\infty$, for $\lambda \in \Omega(a)$. Use the spectral equation,
\[
u_{xx} = \frac{1}{D(\varphi)} \left( b_{1,a}(x) u_x + b_{0,a}(x) u\right) - \lambda u,
\]
and the boundedness of the coefficients on the non-degenerate side to observe that
\begin{equation}
\label{Nddecayuxx}
|u_{xx}| \leq C (1 + |\lambda|) e^{\Re \mu_1^+(\lambda) x} = C_2(\lambda) e^{\Re \mu_1^+(\lambda) x} \to 0,
\end{equation}
as $x \to +\infty$, for some $C_2 = C_2(\lambda) > 0$. Thus, $u_{xx}$ also decays exponentially fast. We use this information to prove the following 
\begin{lemma}
\label{lemwgoodNdplus}
Given $a > 0$ and any fixed $\lambda \in \Omega(a)$, if $u \in H^2$ is a solution to the spectral equation $(\cL_a - \lambda) u = 0$ then the function $w(x) = \Theta(x) u(x)$ belongs to $H^2(x_0, \infty)$ for $x_0 \gg 1$ sufficiently large, where $\Theta = \Theta(x)$ is defined in \eqref{defTheta}.
\end{lemma}
\begin{proof}
It suffices to show that $w$ and its derivatives decay sufficiently fast as $x \to \infty$. First denote
\[
\Re \mu_1^+(\lambda) = a - \frac{c}{2 D(\alpha)} - r^+(\lambda) < 0,
\]
where
\[
r^+(\lambda) := \frac{1}{2}\Re \zeta^+(\lambda)^{1/2} = \frac{1}{2\sqrt{2}} \sqrt{\Re \zeta^+(\lambda) + |\zeta^+(\lambda)| \, } > 0.
\]
Now, we estimate
\[
\begin{aligned}
\Theta(x) e^{\Re \mu_1^+(\lambda)x} &= e^{-\theta(x)} e^{\Re \mu_1^+(\lambda)x} \\
&= \exp \left(\frac{c}{2} \int_{x_0}^{x}\frac{ds}{D(\varphi(s))} - a(x-x_0) \right)  \exp \left( \Big( a - \frac{c}{2D(\alpha)} \Big) x_0 \right) \times \\& \quad \times \exp \left( \Big( a - \frac{c}{2D(\alpha)} \Big) (x-x_0) \right)  \exp \big( - r^+(\lambda) x_0 \big) \times \\ & \quad \times \exp \big( - r^+(\lambda) (x-x_0) \big) \\ 
&\leq  C_0(\lambda) \exp \left(\frac{c}{2} \int_{x_0}^{x}\frac{ds}{D(\varphi(s))} - \frac{c}{2D(\alpha)}(x-x_0) \right) \exp \big( - r^+(\lambda) (x-x_0) \big),
\end{aligned}
\]
for some $C_0 = C_0(\lambda) > 0$. 

Let us define
\[
\chi(x) := \frac{c}{2} \int_{x_0}^{x} \left( \frac{1}{D(\varphi(y))} - \frac{1}{D(\alpha)} \right) \, dy - r^+(\lambda) (x - x_0),
\]
so that
\[
\Theta(x) e^{\Re \mu_1^+(\lambda)x} = e^{-\theta(x)} e^{\Re \mu_1^+(\lambda)x} \leq C_0(\lambda) e^{\chi(x)},
\]
for fixed $\lambda \in \Omega(a)$. From Lemma \ref{lemdecayNd} we have that
\[
\alpha - \varphi(x) = |\alpha - \varphi(x)| \leq C e^{-\eta x},
\]
as $x \to +\infty$. Thus, after a Taylor expansion that reads
\[
\frac{1}{D(\varphi)} - \frac{1}{D(\alpha)} = \frac{D'(\alpha)}{D(\alpha)^2} (\alpha - \varphi) + O(|\alpha - \varphi|^2) \leq C_1 (\alpha - \varphi),
\]
for $\varphi \uparrow \alpha^-$, we arrive at
\[
\frac{c}{2} \int_{x_0}^{x} \left( \frac{1}{D(\varphi(y))} - \frac{1}{D(\alpha)} \right) \, dy \leq \frac{c C_1}{2 \eta} \big(e^{- \eta x_0} - e^{- \eta x}\big) = C_2 \big(e^{- \eta x_0} - e^{- \eta x} \big),
\]
for $x > x_0 \gg 1$, sufficiently large. Hence,
\[
\chi(x) \leq C_2 \big(e^{- \eta x_0} - e^{- \eta x} \big) - r^+(\lambda) (x - x_0) \leq C_2 e^{- \eta x_0} - r^+(\lambda) (x - x_0),
\]
yielding
\[
e^{\chi(x)} = \exp( C_2 e^{- \eta x_0} ) e^{- r^+(\lambda) (x - x_0)} \leq C_3 e^{- r^+(\lambda) (x - x_0)}.
\]
Thus, we have proved the decay estimate
\[
\Theta(x) e^{\Re \mu_1^+(\lambda)x} = e^{-\theta(x)} e^{\Re \mu_1^+(\lambda)x} \leq C(\lambda) e^{-r^+(\lambda)(x-x_0)} \to 0,
\]
as $x \to +\infty$ for some $C(\lambda) > 0$ with $\lambda \in \Omega(a)$ fixed. 

Now, since 
\[
\begin{aligned}
w&= e^{-\theta} u, \\
w_x &= e^{-\theta} \big( u_x - \theta_x u\big),\\
w_{xx} &= e^{-\theta} \big( u_{xx} - 2 \theta_x u_x + (\theta_x^2 - \theta_{xx})u),
\end{aligned}
\]
and noticing that 
\[
\theta_x = a - \frac{c}{2 D(\varphi)}, \quad \theta_{xx} = \frac{c D'(\varphi)\varphi_x}{2 D(\varphi)^2},
\]
are uniformly bounded on the non-degenerate side (as $D(\varphi) \to D(\alpha) > 0$), we apply estimates \eqref{Nddecayuux} and \eqref{Nddecayuxx} to arrive at
\[
|w|, |w_x|, |w_{xx}| \leq C(\lambda) e^{-r^+(\lambda)x} \to 0,
\]
as $x \to +\infty$. This shows that $w \in H^2(x_0, +\infty)$, as claimed.
\end{proof}

\subsubsection{Behavior at $-\infty$}

On the degenerate side, however, we lose hyperbolicity of the asymptotic matrix coefficient and we are not able to apply the Gap Lemma. To guarantee that $w \in H^2(-\infty, x_0)$ we need to impose a further condition on $a \in \R$. This is the content of the following
\begin{lemma}
\label{lemwgoodNdminus}
Let $a \in \R$ satisfy the condition
\begin{equation}
\label{condaNd}
0 < a_1(\alpha) < a < \frac{c}{2 D(\alpha)} < a_2(\alpha).
\end{equation}
Then for any solution $u \in H^2$ to the spectral equation $(\cL_a - \lambda) u = 0$ with $\lambda \in \Omega(a)$ the function $w(x) = \Theta(x) u(x)$ belongs to $H^2(-\infty, x_0)$ for any $x_0 < 0$, $|x_0| \gg 1$ sufficiently large.
\end{lemma}
\begin{proof}
We know that $D(\alpha) > D(\varphi(x)) > 0$ for all $x \in \R$. Therefore, condition \eqref{condaNd} clearly implies that
\[
\int_{x_0}^x \Big( \frac{c}{2 D(\varphi(y))} - a\Big) \, dy < 0,
\]
if $x < 0$, inasmuch as $c/2D(\varphi) - a > 0$. This shows that $0 < \Theta(x) < 1$ if $x < x_0$ for $x_0$ fixed but arbitrary, which includes the case when $x \to -\infty$. Therefore, 
\[
|w(x)|^2 = |\Theta(x)|^2 |u(x)|^2 \leq |u(x)|^2, \quad \text{for all } x < x_0,
\]
yielding $w \in L^2(-\infty, x_0)$. Now, $w_x = \Theta_x u + \Theta u_x$. Since $u \in H^2$ it is then clear that $\Theta u_x \in L^2(-\infty, x_0)$. Let us now compute
\[
\Theta_x = \Big( \frac{c}{2 D(\varphi(x))} - a\Big) \Theta(x).
\]
On the degenerate side $D(\varphi)$ decays to zero as $\varphi \to 0^+$. Moreover, we have proved that $\varphi \to 0^+$ as $x \to -\infty$ at rate $O(\exp (|f'(0)|x/c))$ (see Lemma \ref{lemdecayNd}). Therefore, we make expansions around $\varphi = 0$ of the form
\begin{equation}
\label{Dexp}
\begin{aligned}
D(\varphi) &= D'(0) \varphi + O(\varphi^2),\\
\frac{1}{D(\varphi)} &=\frac{1}{D'(0) \varphi} + O(1), 
\end{aligned}
\qquad \text{as } \; \varphi \to 0^+.
\end{equation}
Hence, from $0 < \varphi(x) \leq C \exp (|f'(0)|x/c)$ when $x \to -\infty$ and some $C > 0$ we obtain the bound
\[
\int_{x_0}^x \frac{dy}{\varphi(y)} \leq C_1 \int_{x_0}^x \exp \big(- |f'(0)|y/c \big) \, dy,
\]
for $x < x_0 < 0$, $|x_0| \gg 1$ large enough and some uniform $C_1 > 0$. Substituting the expansion for $D(\varphi)^{-1}$ for $x < x_0 < 0$ with $|x_0| \gg 1$ we arrive at
\[
\begin{aligned}
\frac{c}{2} \int_{x_0}^x \frac{dy}{D(\varphi(y))} - a(x - x_0) &\leq \frac{c}{2} \int_{x_0}^x \Big( \frac{1}{D'(0) \varphi(y)} + C_2 \Big) \, dy  - a(x-x_0)\\
&\leq \frac{c C_1}{2 D'(0)} \int_{x_0}^x \exp \big( - |f'(0)|y/c \big) \, dy + C_3(x-x_0)\\
&= \frac{c^2 C_1}{2 D'(0) |f'(0)|} \big( e^{-|f'(0)|x_0/c} - e^{-|f'(0)|x/c}\big) - C_3 |x-x_0|,
\end{aligned}
\]
for all $x < x_0 \ll -1$ and with uniform constants $C_j > 0$. This implies that
\[
0 < \Theta(x) \leq \exp \Big( - C_4 e^{-|f'(0)|x_0/c} - C_3 |x-x_0|\Big) \leq \exp \Big( - C_4 e^{-|f'(0)|x_0/c}\Big) \to 0,
\]
as $x \to -\infty$ for some $C_4 > 0$.

Now, in view that $\varphi \sim C_0 \exp (|f'(0)|x/c)$ asymptotically as $x \to -\infty$ (see \eqref{asintotaNd}), use
\[
\frac{1}{C_0} \varphi(x) e^{-|f'(0)|x/c} \to 1, \quad \text{as } \; x \to - \infty,
\]
and the expansion \eqref{Dexp} to show that
\[
|\Theta_x| \leq \left| \frac{c}{2 D'(0) \varphi(x)} + a \right| \exp \Big( - C_4 e^{-|f'(0)|x/c} \Big) \leq \frac{C e^{-|f'(0)|x/c}}{\exp \big( C_4 e^{-|f'(0)|x/c}\big)}  \to 0,
\]
as $x \to -\infty$. Therefore, we readily conclude that $w_x = \Theta_x u + \Theta u_x \in L^2(-\infty, x_0)$ for $x_0 \ll -1$.

In a similar fashion it can be proved that $\Theta_x D(\varphi)^{-1}$ and $\varphi_x \Theta D(\varphi)^{-2}$ are bounded above as $x \to -\infty$, due to the fast decay of $\Theta = O(\exp( - C e^{-kx}))$ as $x \to -\infty$, with appropriate constants $C, k > 0$. Details are left to the reader. This analysis shows, in turn, that $w_{xx} = \Theta_{xx} u + 2 \Theta_x u_x + \Theta u_{xx} \in L^2(-\infty, x_0)$ as well. We conclude that $w \in H^2(-\infty, x_0)$, as claimed.
\end{proof}

\subsection{Point spectral stability}
\label{secptspNd}

The decay estimates of Section \ref{secdecaystr}, as well as the establishment of inequality \eqref{condaNd} for the exponential weight, are the key ingredients to show that the point spectrum of the conjugated operator is stable. First we verify that for any $a \in \R$ satisfying \eqref{condaNd} the eigenfunction of the linearized operator around the wave associated to the zero eigenvalue transforms, under multiplication by $e^{ax}$, into an eigenfunction for the conjugated operator with same eigenvalue.

\begin{lemma}
\label{lemzeroefNd}
Let $\varphi$ be a degenerate, monotone Nagumo front traveling with speed $c > \bbc(\alpha)$ (front of type I in Proposition \ref{propstructure}), and let $a \in \R$ be such that condition \eqref{condaNd} holds. Then the function
\[
\phi(x) := e^{ax} \varphi_x(x), \qquad x \in \R,
\]
belongs to $H^2(\R)$. Moreover, $\lambda = 0 \in \ptsp(\cL_a)_{|L^2}$ with associated eigenfunction $\phi \in \ker (\cL_a) \subset H^2$.
\end{lemma}
\begin{proof}
On the non-degenerate side as $\varphi \to \alpha$ when $x \to \infty$, the front and its derivatives decay asymptotically as (see Lemmata \ref{lemdecayNd} and \ref{lemC4})
\[
|\partial_x^j(\varphi - \alpha) | \leq C e^{- a_2(\alpha)x}, 
\]
for $j = 0,1,2,3$. By hypothesis, $a \in \R$ satisfies \eqref{condaNd}. Thus, $a > a(\alpha)$ and then it is clear that
\[
\begin{aligned}
\phi &=  e^{ax} \varphi_x,\\
\phi_x &= e^{ax} ( a \varphi_x + \varphi_{xx}),\\
\phi_{xx} &= e^{ax} (a^2 \varphi_x + 2a \varphi_{xx} + \varphi_{xxx}),
\end{aligned}
\]
decay exponentially to zero when $x \to \infty$ with rate $|\partial_x^j \phi | \leq Ce^{- \nu x}$ for all $j$ and where $\nu := a_2(\alpha) - a > 0$.

On the degenerate side as $\varphi \to 0$ when $x \to -\infty$ the front and its derivatives decay exponentially as (see Lemma \ref{lemdecayNd})
\[
|\partial_x^j \varphi | \leq C e^{|f'(0)|x/c},
\]
for $j = 0,1,2,3.$ Thus, since $a > 0$ it is then straightforward to observe that $\phi$ and its derivatives decay exponentially as well, in view that
\[
|e^{ax}\varphi_x|, \, |e^{ax}\varphi_{xx}|, \, |e^{ax}\varphi_{xxx}| \leq C e^{|f'(0)|x/c} \to 0,
\]
as $x \to -\infty$. Since the front is at least of class $C^4$, the exponential decay of $\phi$ and its derivatives at $x = \pm \infty$ is sufficient to conclude that $\phi \in H^2(\R)$.

Finally, notice that $\lambda = 0 \in \ptsp(\cL_a)$ with associated eigenfunction $\phi \in H^2$ because $\cL_a \phi = e^{ax} \cL e^{-ax} \phi = e^{ax} \cL \varphi_x = 0$. This proves the lemma.
\end{proof}

\begin{theorem}[point spectral stability]
\label{thmsptNd}
Let $\varphi$ be a diffusion-degenerate monotone Nagumo front traveling with speed $c > \bbc(\alpha)$ (of type I in Proposition \ref{propstructure}). If $a \in \R$ satisfies inequality \eqref{condaNd} then the conjugated linearized operator around the wave is point spectrally stable in $L^2$. More precisely,
\[
\ptsp(\cL_a)_{|L^2} \subset (-\infty,0].
\] 
\end{theorem}
\begin{proof}
By definition of $\ptsp$ (see Definition \ref{defspecd}), if $\lambda \in \ptsp(\cL_a)_{|L^2}$ then there exists $u \in H^2$ such that $(\cL_a - \lambda) u = 0$. Now let $\lambda \in \ptsp(\cL_a)_{|L^2} \cap \Omega(a)$. For each $x \in \R$ let us define $w = w(x)$ as in \eqref{defofchangew}, where $u \in H^2$ is a solution to the spectral equation. Pick $x_0 \in \R$, $x_0 \gg 1$ sufficiently large. Hence, if $a \in \R$ satisfies condition \eqref{condaNd} then we apply Lemmata \ref{lemwgoodNdminus} and \ref{lemwgoodNdplus} to reckon
\[
\begin{aligned}
w(x) &= \exp \left( \frac{c}{2} \int_{x_0}^x \frac{dy}{D(\varphi(y))} - a(x-x_0) \right) u(x) \in H^2(x_0, \infty),\\
&= C_0 \exp \left( \int_{-x_0}^x \left( \frac{c}{2 D(\varphi(y))} - a \right) dy \right) u(x) \in H^2(-\infty, -x_0).
\end{aligned}
\]

Since for any finite $x_0 \in \R$ we clearly have $w \in H^2(-x_0, x_0)$, we conclude that $w \in H^2(\R)$ and Assumption \ref{assumcrucial} holds. Moreover, it is to be observed that under condition \eqref{condaNd} we have
\[
\Omega(a) = \{ \Re \lambda > - \mu_0 \},
\]
where $\mu_0 > 0$ (see Lemma \ref{lemsdNd}). Therefore, $0 \in \Omega(a)$. Also thanks to Lemma \ref{lemzeroefNd}, $\phi = e^a \varphi_x$ is the eigenfunction of $\cL_a$ associated to the zero eigenvalue, and Assumption \ref{assumcrucial} is also valid for $\psi = e^{-\theta} \phi$. Henceforth, all the hypotheses of Lemma \ref{lembee} are satisfied. We conclude that for each $\lambda \in \ptsp(\cL_a)_{|L^2} \cap \Omega(a)$ the basic energy estimate \eqref{basicee} holds yielding $\ptsp(\cL_a)_{L^2} \subset (-\infty,0]$, as claimed.
\end{proof}

\subsection{Location of $\sigma_\pi(\cL_a)_{|L^2}$}
\label{seclocapp}

In this section we locate the subset of the approximate spectrum, namely $\sigma_\pi(\cL_a)_{|L^2}$ in the stable half plane. For a given $a \in \R$ satisfying \eqref{condaNd} let us write the conjugated operator \eqref{conjopL}, $\cL_a : \cD(\cL_a) = H^2 \subset L^2 \to L^2$, as
\[
\cL_a u = b_2(x) u_{xx} + b_1(x) u_x + b_0(x) u, \qquad u \in H^2,
\]
where
\begin{align*}
b_2(x) &= D(\varphi),\\b_1(x) &= D(\varphi)_x - 2aD(\varphi) + c,\\
b_0(x) &= a^2 D(\varphi) - 2a D(\varphi)_x - ac + D(\varphi)_{xx} + f'(\varphi).
\end{align*}
Let us denote,
\[
b_j^\pm := \lim_{x \to \pm \infty} b_j(x), \qquad j = 0,1,2,
\]
yielding
\[
\begin{aligned}
b_2^+ &= D(\alpha), &  b_2^- &= 0,\\
b_1^+ &= c - 2 a D(\alpha), &  b_1^- &= c,\\
b_0^+ &= a^2 D(\alpha) - ac + f'(\alpha), &  b_0^- &= f'(1) - ac.\\
\end{aligned}
\]

\begin{lemma}
\label{lemsigmapiNd}
Suppose that $a \in \R$ satisfies condition \eqref{condaNd}. Then the $\sigma_{\pi}(\cL_a)_{|L^2}$ spectrum of the linearized operator around a degenerate Nagumo front of type I as in Proposition \ref{propstructure}, is stable:
  \begin{equation*}
    \sigma_{\pi}(\cL_a)_{|L^2} \subset \{ \lambda \in \C \, : \, \Re \lambda \leq 0 \}.
  \end{equation*}
\end{lemma}
\begin{proof}
Let $\lambda \in \sigma_{\pi}(\cL_a)_{|L^2}$ be fixed. Then, from Definition \ref{defspecd} we know that the range of $\cL_a - \lambda$ is not closed and there exists a singular sequence $u_n \in \cD(\cL) = H^2$ with $\| u_n \|_{L^2} = 1$, for all $n \in \N$, such that $(\cL_a - \lambda) u_n \to 0$ in $L^2$ as $n \to \infty$. Since $L^2$ is a reflexive space, this sequence can be chosen so that $u_n \rightharpoonup 0$ in $L^2$ (see Remark \ref{remponla}). 

We claim that there exists a subsequence, which we still denote by $u_n$, such that $u_n \to 0$ in $L^2_\textrm{loc}$ as $n \to \infty$. Indeed, fix an open bounded interval $I$ and let $f_n := (\cL_a - \lambda)u_n$. Take the complex $L^2$-product of $(\cL_a - \lambda)u_n = f_n$ with $u_n$, integrate by parts and take the real part. The result is
\begin{equation}
 \label{Ndlambda}
\Re \lambda = - \Re \langle f_n, u_n \rangle_{L^2} - \langle b_2(x) \partial_x u_n, \partial_x u_n \rangle_{L^2} + \langle (b_0(x) - \tfrac{1}{2} \partial_x b_1(x)) u_n, u_n \rangle_{L^2}.
\end{equation}
This equation together with the hypothesis on $u_n$, $f_n$, $\lambda$ fixed and the strict positivity of $b_2 = D(\varphi)$ on $I$, imply that $u_n$ is bounded in the space $H^1(I)$. Therefore, by the Rellich-Kondrachov theorem there exists a subsequence such that $u_n \to 0$ in $L^2(I)$. By using a standard diagonal argument in increasing intervals, we arrive at a subsequence $u_n$ such that $u_n \to 0$ in $L^2_\textrm{loc}$ as $n \to +\infty$.

Now, since
\[
\begin{aligned}
b_0(x) - \tfrac{1}{2} \partial_x b_1(x) &= a^2 D(\varphi) - a D(\varphi)_x - ac + \tfrac{1}{2} D(\varphi)_{xx} + f'(\varphi) \\& \to a^2 D(u_\pm) -ac + f'(u_\pm),
\end{aligned}
\]
as $x \to \pm \infty$, thanks to the fact that $a \in \R$ satisfies \eqref{condaNd} we can choose $R>0$ sufficiently large such that
  \begin{equation*}
    b_0(x) - \tfrac{1}{2} \partial_x b_1(x)  = \tfrac{1}{2} D(\varphi)_{xx} + f'(\varphi) < 0 \quad \text{for } |x| \ge R.
  \end{equation*}
Hence, from $b_2(x) = D(\varphi) \geq 0$ and \eqref{Ndlambda} we get
\[
\begin{aligned}
\Re \lambda &\leq |\langle f_n, u_n \rangle_{L^2}| + \int_{-R}^R \!(b_0(x) - \tfrac{1}{2} \partial_x b_1(x)) |u_n|^2 \, dx + \int_{|x|\geq R} \!\!(b_0(x) - \tfrac{1}{2} \partial_x b_1(x)) |u_n|^2 \, dx \\
&\leq \| (\cL_a - \lambda) u_n \|_{L^2} + C \| u_n \|_{L^2(-R,R)} \, \to 0,
\end{aligned}
\]
as $n \to \infty$, thanks to boundedness of the coefficients, $\| u_n \|_{L^2} = 1$, $(\cL_a - \lambda) u_n \to 0$ in $L^2$ and to the convergence of $u_n$ to zero in $L^2_\textrm{loc}$. The lemma is proved.
\end{proof}

\subsection{Proof of Main Theorem \ref{mainthmNd}}
Under the hypotheses of Theorem \ref{mainthmNd}, choose any $a \in \R$ such that condition \eqref{condaNd} holds. More precisely,
\[
0 < a_1(\alpha) < a < \frac{c}{2D(\alpha)} < a_2(\alpha).
\]
Then apply Theorem \ref{thmsptNd} and Lemmata \ref{lemdspecloc} and \ref{lemsigmapiNd} in order to obtain
\[
\begin{aligned}
\sigma(\cL)_{|L^2_a} = \sigma(\cL_a)_{|L^2} = \ptsp(\cL_a)_{|L^2} \cup \sigma_\delta(\cL_a)_{|L^2} \cup \sigma_\pi(\cL_a)_{|L^2} \subset \{ \lambda \in \C \, : \, \Re \lambda \leq 0\},
\end{aligned}
\]
as claimed. This proves the Theorem.
\qed

\section*{Acknowledgements}

We are grateful to an anonymous referee for her/his valuable comments which significantly improved both the quality and the scope of the paper. We also thank Raffaele Folino for useful discussions. The work of L. F. L\'opez R\'{\i}os was supported by a Postdoctoral Fellowship from DGAPA-UNAM. R. G. Plaza was partially supported by DGAPA-UNAM, program PAPIIT, grant IN-100318.

%

\def\cprime{$'$}

\end{document}